\documentclass[11pt]{amsart}
\usepackage{mathpazo}
\usepackage{geometry}              
\usepackage{graphicx}
\usepackage[parfill]{parskip}
\usepackage{amssymb}
\usepackage[all]{xy}
\usepackage{color}
\usepackage{tikz}
\usetikzlibrary{arrows,decorations.markings,fit,matrix,positioning,calc,shapes.symbols,shapes}
\usepackage{dsfont}

\newcommand{\tc}[2]{\textcolor{#1}{#2}}
\definecolor{dmagenta}{rgb}{.5,0,.5} 
\definecolor{dred}{rgb}{.5,0,0} 
\definecolor{dgreen}{rgb}{0,.5,0} 
\definecolor{dblue}{rgb}{0,0,0.5} 
\definecolor{black}{rgb}{0,0,0} 
\definecolor{vdgreen}{rgb}{0,.3,0} 
\definecolor{vdred}{rgb}{.3,0,0} 
\definecolor{red}{rgb}{1,0,0} \newcommand{\red}[1]{\tc{red}{#1}}

\newcommand{\lplus}{{\mathfrak{h}}}        

\newcommand\cG{{\mathcal{G}}}
\newcommand\cGB{{\mathcal{G}^b}}
\newcommand\cO{{\mathcal{O}}} 

\newcommand{\Lie}{\mathsf{Lie}}  
\newcommand{\Q}{\F}

\newcommand{\F}{{\mathds{k}}}
\newcommand{\Z}{{\mathbb{Z}}}

\newcommand{\sJ}{\mathsf{J}}
\newcommand{\hairy}{\mathcal H}  

\newcommand{\ext}{\bigwedge\nolimits}

\newcommand{\id}{\mathrm{Id}}   
\DeclareMathOperator{\Tr}{Tr}  
\DeclareMathOperator{\SP}{Sp}  
\DeclareMathOperator{\Sym}{Sym}  
\newcommand{\sym}[1]{\mathfrak{S}_{#1}}
\DeclareMathOperator{\Aut}{Aut}  
\DeclareMathOperator{\Out}{Out} 
\DeclareMathOperator{\Tor}{Tor}  
\DeclareMathOperator{\End}{End}  
\DeclareMathOperator{\Mod}{Mod}  
\DeclareMathOperator{\im}{im}  
\DeclareMathOperator{\SL}{SL}
\DeclareMathOperator{\GL}{GL}  
\DeclareMathOperator{\ad}{ad}
\let\ker\undefined
\DeclareMathOperator{\ker}{ker}  

\let\sl\undefined
\DeclareMathOperator{\sl}{\mathfrak{sl}}   
\DeclareMathOperator{\Hom}{Hom}
\DeclareMathOperator{\HomH}{Hom^\mathcal{H}}

\newcommand{\SF}[1]{{\mathbb S}_{#1}}  

\newcommand{\arity}[1]{{(\!(#1)\!)}}
\newcommand{\sL}{{\mathsf L}}
\newcommand{\sD}{{\mathsf D}}

\newcommand{\cT}{\mathcal T}

\newcommand{\la}{\langle}
\newcommand{\ra}{\rangle}

\newtheorem{proposition}{Proposition}[section]
\newtheorem{theorem}[proposition]{Theorem}
\newtheorem{lemma}[proposition]{Lemma}
\newtheorem{claim}[proposition]{Claim}

\newtheorem{corollary}[proposition]{Corollary}
\theoremstyle{remark}

\theoremstyle{definition}
\newtheorem{definition}[proposition]{Definition}
\newtheorem{remark}[proposition]{Remark}

\newtheorem*{example}{Example}
\newtheorem*{examples}{Examples}
\newtheoremstyle{red}{3pt}{3pt}{\color{red}}{}{\itshape}{.}{.5em}{}
\theoremstyle{red}

\makeatletter
\def\Ddots{\mathinner{\mkern1mu\raise\p@
\vbox{\kern7\p@\hbox{.}}\mkern2mu
\raise4\p@\hbox{.}\mkern2mu\raise7\p@\hbox{.}\mkern1mu}}
\makeatother

\tikzset{
empty/.style={inner sep=-1pt,minimum size=0mm},
emptyunit/.style={circle,draw=white,fill=white,thick, inner sep=0pt,minimum size=2mm},
emptyantipode/.style={circle,draw=white,fill=white,thick, inner sep=0pt,minimum size=4mm},
operad/.style={circle,draw=red!50,fill=red!20,thick, inner sep=0pt,minimum size=6mm},
hopf/.style={signal, signal to=east, signal from=west,draw=brown!50,fill=brown!20,thick, inner sep=2pt,minimum size=6mm},
lhopf/.style={signal, signal to= west, signal from= east,draw=brown!50,fill=brown!20,thick, inner sep=2pt,minimum size=6mm},
antipode/.style={circle,draw=purple!50,fill=purple!20,thick, inner sep=0pt,minimum size=4mm},
unit/.style={circle,draw=black,fill=white,thick, inner sep=0pt,minimum size=2mm},
break/.style={inner sep=0pt,minimum size=5mm},
block/.style={draw=blue,fill=blue!20,thick,inner sep=10pt}
outer/.style={}
}

\title{Hopf algebras and invariants of the Johnson cokernel
}
\author{Jim Conant, Martin Kassabov}
\begin{document}
\begin{abstract}
We show that if $H$ is a cocommutative Hopf algebra, then there is a natural action of $\Aut(F_n)$ on $H^{\otimes n}$ which induces an $\Out(F_n)$ action on a quotient $\overline{H^{\otimes n}}$. In the case when $H=T(V)$ is the tensor algebra, we show that the invariant $\Tr^C$ of the cokernel of the Johnson homomorphism studied in~\cite{C} projects to take values in $H^{vcd}(\Out(F_n);\overline{H^{\otimes n}})$. We analyze the $n=2$ case, getting large families of obstructions generalizing the abelianization obstructions of~\cite{CKV2}.
\end{abstract}
\maketitle

\section{Introduction}
\label{sec:intro}
The mapping class group $\Mod(g,1)$ of a surface of genus $g$ with one boundary component carries the \emph{Johnson filtration} $$
\Mod(g,1)= \mathbb J_0 \supset \mathbb J_1\supset \mathbb J_2 \supset \cdots
$$
defined by letting $\mathbb J_s$ be the those elements of $\Mod(g,1)$ which act trivially on $\pi_1(\Sigma_{g,1})$ modulo the $(s+1)$st term of the lower central series. The associated graded, tensored with  a characteristic $0$ field $\F$, has the structure of a Lie algebra $\sJ$.
The Johnson homomorphism is an embedding of Lie algebras
$\tau\colon \sJ\to\sD^+$, where $\sD^+=\sD^+$ has several natural algebraic definitions (see section~\ref{sec:cokernel}). Both $\sJ$ and $\sD^+$ have decompositions into $\SP$ modules, stemming from the symplectic structure on $H_1(\Sigma_{g,1};\F)$, and a basic problem in this area is to identify the (stable) $\SP$-decomposition of $\sJ$. (The decomposition stabilizes as $g\to\infty$.) In some sense the decomposition of the Lie algebra $\sD^+$ is ``easy," given its straightforward definition, so the problem can be reduced to studying the cokernel of the Johnson homomorphism $\mathsf C=\sD^+/\tau \sJ$, see~\cite{MoSurvey} for an overview.

One way to define the Lie algebra $\sD^+$ is as a vector space of unitrivalent trees with leaves labeled by $V=H_1(\Sigma_{g,1},\F)$, modulo IHX and orientation relations. The bracket is defined by summing over joining two trees along a pair of univalent vertices --multiplying by the contraction of the vertex labels--in all possible ways. By a theorem of Hain~\cite{Ha}, $\im\tau\subset\sD^+$ is generated (in the stable range) as a Lie subalgebra by trees with one trivalent vertex (called \emph{tripods}).

In~\cite{C} a ``trace" map $\Tr^C\colon \sD^+\to \oplus_{n\geq 1}\Omega_n(V)$ was constructed. The spaces $\Omega_n(V)$ are generated by graphs formed by adding $n$ ``external" oriented edges to a labeled tree from $\sD^+$, while the trace map is defined by adding sets of edges to a tree in all possible ways, multiplying by contractions of labeling coefficients. A suitable quotient is taken to ensure that $\Tr^C$ vanishes on iterated brackets of tripods. The spaces $\Omega_n(V)$ detect many families of cokernel elements~\cite{C,C2}, and could plausibly even be used to give a complete description of the cokernel.  However the combinatorics quickly becomes complicated, and it becomes desirable to have a more conceptual description of these obstruction spaces. To that end, in the current paper, we construct, for $n\geq 2$, surjections
$$
\Omega_n(V)\twoheadrightarrow H^{2n-3}(\Out(F_n);\overline{T(V)^{\otimes n}}),
$$
where the coefficient module is a certain quotient of the $n$-th tensor power of the tensor algebra $T(V)$ and the action of $\Out(F_n)$ is defined via the Hopf algebra structure on $T(V)$. Indeed, we show that for any cocommutative Hopf algebra $H$, $\Aut(F_n)$ acts in a natural way on $H^{\otimes n}$.   When $H=\Sym(V)$, it reduces to the standard action of $\GL_n(\Z)$ on $\Sym(V)^{\otimes n}\cong \Sym(V\otimes \F^n)$. The module $\overline{H^{\otimes n}}$ is an appropriate quotient on which inner automorphisms act trivially, so that one gets an $\Out(F_n)$ action. See section~\ref{sec:action_on_Hn} for complete details. As far as we know, it is a novel construction. In particular, it does not factor through $\GL_n(\Z)$ and even in the case when $n=2$ and $\Out(F_2)=\GL_2(\Z)$, it does not extend to an action of $\SL^{\pm}_2(\F)$.

In preparation for stating the main theorem, we note that, as observed in~\cite{C}, a result of~\cite{CKV2} implies that $\Tr^C$ surjects onto each $[\lambda]_{\SP}\subset [\lambda]_{\GL}$ in the $\GL$-decomposition of $\Omega_n(V)$, for $g$ sufficiently large with respect to $|\lambda|$.

As a result, we have the following theorem. (The $n=1$ case is calculated  is in \cite{C}.)

\begin{theorem}
\label{thm:intro}
For all $n\geq 2$, and for $g$ sufficiently large compared to the degree, there is an invariant defined on the Johnson cokernel, taking values in
$ H^{2n-3}(\Out(F_n);\overline{T(V)^{\otimes n}})$.
Moreover, for large enough $g$ compared to $|\lambda|$, the invariant surjects onto each $[\lambda]_{\SP}\subset [\lambda]_{\GL}$ in the $\GL$-decomposition of the image.
\end{theorem}

Thus, we have a satisfying conceptual construction of modules  $H^{2n-3}(\Out(F_n);\overline{T(V)^{\otimes n}})$ which obstruct being in the image of the Johnson homomorphism. These are less mysterious than $\Omega_n(V)$, which are defined using generators and relations (although see \cite{C2} for progress in calculating these). One might hope that they capture all of $\Omega_n(V)$, but work in \cite{C2} shows that $\Omega_2(V)$ is strictly larger than $ H^{1}(\Out(F_2);\overline{T(V)^{\otimes 2}})$, and it seems likely that this holds for $n>2$ as well.

\subsection{Comparison to existing obstructions}
It has long been known, following again from Hain's theorem, that $\mathsf C$ surjects onto the abelianization of $\sD^+$ for degree $\geq 2$. In~\cite{CKV1}, $\sD^+$ was shown to embed in
$$
\ext^3 V\oplus \bigoplus_{k\geq 1}\Sym^{2k+1}(V)\oplus \bigoplus_{n\geq 2}H^{2n-3}(\Out(F_n);\Sym(V)^{\otimes n}),
$$
and surjects onto each $[\lambda]_{\SP}\subset [\lambda]_{\GL}$. In fact one recovers the abelianization obstructions from Theorem~\ref{thm:intro} by applying the map $T(V)\to \Sym(V)$ to the coefficient modules. In \cite{CKV1}, the $n=2$ part of the abelianization is calculated, while the even degree part in $n=3$ is calculated in \cite{C2}.

As shown in~\cite{C}, the module $\Omega_1(V)$ picks up the trace map constructed by Enomoto and Satoh~\cite{ES}.
Moreover, the Enomoto-Satoh trace map is a lift of Morita's original trace map~\cite{Mo}, which is an invariant of the $n=1$ part of the abelianization of $\sD^+$.
All of these obstructions are organized in Figure~\ref{fig:chart}. The bottom row is given by the abelianization obstructions and the top by the obstructions from~\cite{C}. Sitting in between these are the new obstruction modules which are the subject of this paper.

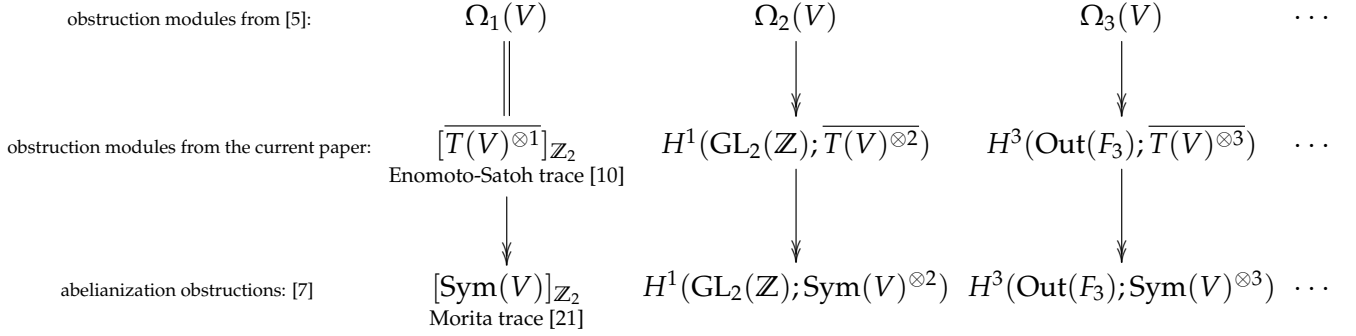
\begin{figure}
$$
\xymatrix@C-2pc{
\text{\tiny obstruction modules from \cite{C}:}&\Omega_1(V)\ar@{=}[d] &{\Omega_2(V)}\ar@{->>}[d] & \Omega_3(V)\ar@{->>}[d]&\cdots\\
\text{\tiny obstruction modules from the current paper:}
&\underset{\text{Enomoto-Satoh trace \cite{ES}}}{[\overline{T(V)^{\otimes 1}}]_{\Z_2}}\ar@{->>}[d]&{H^1(\GL_2(\Z);\overline{T(V)^{\otimes 2}})}\ar@{->>}[d]&{H^3(\Out(F_3);\overline{T(V)^{\otimes 3}})}\ar@{->>}[d]&\cdots\\
{\text{\tiny abelianization obstructions: \cite{CKV1} }}&\underset{\text{Morita trace \cite{Mo}}}{[\Sym(V)]_{\Z_2}}&{H^1(\GL_2(\Z);\Sym(V)^{\otimes 2})}&{H^3(\Out(F_3);\Sym(V)^{\otimes 3})}&\cdots
}
$$
\caption{Obstruction modules for the Johnson cokernel. In section \ref{sec:calculations}, we partially compute $H^1(\GL_2(\Z);\overline{T(V)^{\otimes 2}})$. The module $H^1(\GL_2(\Z);\Sym(V)^{\otimes 2})$ was completely calculated in \cite{CKV1}. Partial computations for $\Omega_2(V), \Omega_3(V)$ and $H^3(\Out(F_3);\Sym(V)^{\otimes 3})$ are made in \cite{C2}, building on the results and methods of the current paper.}
\label{fig:chart}
\end{figure}

\subsection{Explicit computations for $n=2$}
We make explicit computations for the case $n=2$ in section~\ref{sec:calculations}. In particular, let $\sL_{(2)}=V\oplus \ext^2V$ be the free nilpotent Lie algebra of nilpotency class $2$, and let $\overline{\Sym(\sL_{(2)})^{\otimes 2}}$ be the quotient of $\Sym(\sL_{(2)})^{\otimes 2}$ by the image of the adjoint action of $\sL_{(2)}$.
  We show that
there is a surjection
$$
H^1(\Out(F_2);\overline{T(V)^{\otimes 2}})
\twoheadrightarrow
H^1(\GL_2(\mathbb Z);\overline{\Sym(\sL_{(2)})^{\otimes 2}})
$$
where $\GL_2(\mathbb Z)$ acts on
${\Sym(\sL_{(2)})^{\otimes 2}}\cong {\Sym(\sL_{(2)}\otimes \F^2)}$ via the standard action on $\F^2$, which induces the action of $\GL_2(\mathbb Z)$ on $\overline{\Sym(\sL_{(2)})^{\otimes 2}}$.
This latter group can be computed via the methods of~\cite{CKV1}, leading to many families of obstructions
$[\lambda]_{\SP}\otimes \mathcal M_k$ and $[\lambda]_{\SP}\otimes \mathcal S_k,$ where $\mathcal M_k$, $\mathcal S_k$ denote spaces of modular (respectively cusp) forms of weight $k$. The simplest new families of obstructions one gets are
$$
[2k-1,1^2]_{\SP}\otimes \mathcal S_{2k+2}\subset \mathsf{C}_{2k+5}, \text{ and }
 ([2k+1,1^2]_{\SP}\oplus[2k,2,1]_{\SP}\oplus[2k,1^3]_{\SP})\otimes \mathcal M_{2k+2}\subset\mathsf{C}_{2k+7}.
$$

The above surjection exists because the $\GL_2(\Z)$ action on $\overline{T(V)^{\otimes 2}}$ extends to an $\SL^{\pm}_2(\F)$ action modulo commutators of length $3$. However it does not extend modulo commutators of length $4$. This makes the full analysis of
the cohomology group $H^1(\Out(F_2);\overline{T(V)^{\otimes 2}})$ less than straightforward. We hope to pursue this in a future paper.

{\bf Acknowledgements:} We thank Nolan Wallach and Darij Grinburg for helpful discussions.
Jim Conant partially worked on this paper during a visit to Max-Planck-Institut-f\"ur-Mathematik in Summer 2015.
Martin Kassabov was partially supported by  Simons Foundation grant 305181 and NSF grants DMS 0900932 and 1303117.

\section{Hopf algebras} 
Fix a ground field $\F$ of characteristic $0$.
Let $H$ be a cocommutative Hopf algebra with multiplication $m\colon H\otimes H\to H$, comultiplication $\Delta\colon H\to H\otimes H$, unit $\eta\colon \F\to H$, co-unit $\epsilon\colon H\to F$ and antipode $S\colon H\to H$. (See Figure~\ref{graph-calc} for the axioms that these must satisfy.) As we are in the cocommutative setting, the antipode satisfies $S^2=\id$ and $S(ab)=S(b)S(a)$, $\Delta (S(a))=(S\otimes S)\tau \Delta (a)$, where $\tau\colon H\otimes H\to H\otimes H$ is the twist map $a\otimes b\mapsto b\otimes a$.
There is a graphical calculus \`a la Joyal and Street \cite{JS} that can be used to describe the Hopf algebra operations. The basic Hopf algebra operations are depicted in Figure~\ref{graph-calc1}. Graphical translations of the axioms  are depicted in Figure~\ref{graph-calc}.
Our conventions are to read morphisms from left to right and tensor products go up/down.

\begin{figure}
$\underset{\text{Multiplication}}{
\begin{tikzpicture}
\node[empty](aa){\,\,$H$};
\node[empty](bb)[left of=aa]{};
\draw (bb.center) to (aa);
\node[empty](cc)[above left of =bb]{};
\node[empty](ee)[left of=cc]{$H$\,\,};
\node[empty](dd)[below left of =bb]{};
\node[empty](ff)[left of=dd]{$H$\,\,};
\draw (bb.center) to[out=180,in=0] (ee);
\draw (bb.center) to[out=180,in=0] (ff);
\node[empty][above of=bb,node distance=.3cm]{$m$};
\end{tikzpicture}}$
\hspace{5em}
$\underset{\text{Comultiplication}}{
\begin{tikzpicture}
\node[empty](aa){$H$\,\,};
\node[empty](bb)[right of=aa]{};
\draw (bb.center) to (aa);
\node[empty](cc)[above right of =bb]{};
\node[empty](ee)[right of=cc]{\,\,$H$};
\node[empty](dd)[below right of =bb]{};
\node[empty](ff)[right of=dd]{\,\,$H$};
\draw (bb.center) to[out=0,in=180] (ee);
\draw (bb.center) to[out=0,in=180] (ff);
\node[empty][above of=bb,node distance=.3cm]{$\Delta$};
\end{tikzpicture}}$
\hspace{5em}
$\underset{\text{Antipode}}{
\begin{tikzpicture}
\node[empty](aa){$H$\,\,};
\node[antipode](bb)[right of =aa]{S};
\draw (bb) to (aa);
\node[empty](cc)[right of =bb]{\,\,$H$} edge (bb);
\node[empty](ee)[below of =bb]{};
\end{tikzpicture}}$
\\
$\underset{\text{Counit}}{\begin{tikzpicture}
\node[empty](aa){$H$\,\,};
\node[unit](bb)[right of =aa]{} edge (aa);
\node[empty](cc)[right of =bb]{\,\,$\F$} ;
\node[empty](ee)[below of =bb, node distance=.2cm]{};
\node[empty](ff)[above of =bb, node distance=.3cm]{$\epsilon$};
\node[empty](gg)[above of =ff, node distance=.6cm]{};
\end{tikzpicture}}$
\hspace{5em}
$\underset{\text{Unit}}{
\begin{tikzpicture}
\node[empty](aa){$\F$\,\,};
\node[unit](bb)[right of =aa]{};
\node[empty](cc)[right of =bb]{\,\,$H$} edge (bb);
\node[empty](ee)[below of =bb, node distance=.2cm]{};
\node[empty](ff)[above of =bb, node distance=.3cm]{$\eta$};
\node[empty](gg)[above of =ff, node distance=.6cm]{};
\end{tikzpicture}}$
\caption{Hopf algebra operations depicted graphically, read from left to right.}
\label{graph-calc1}
\end{figure}
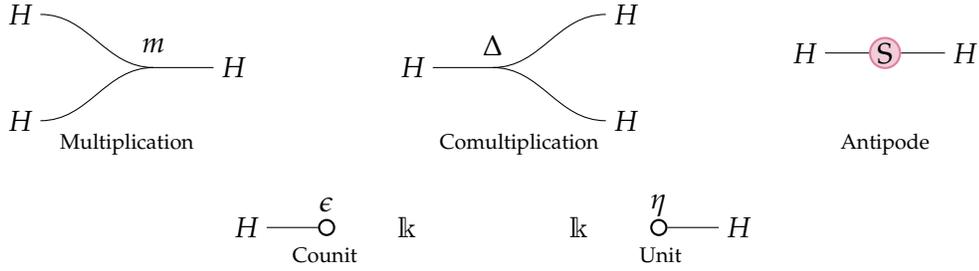

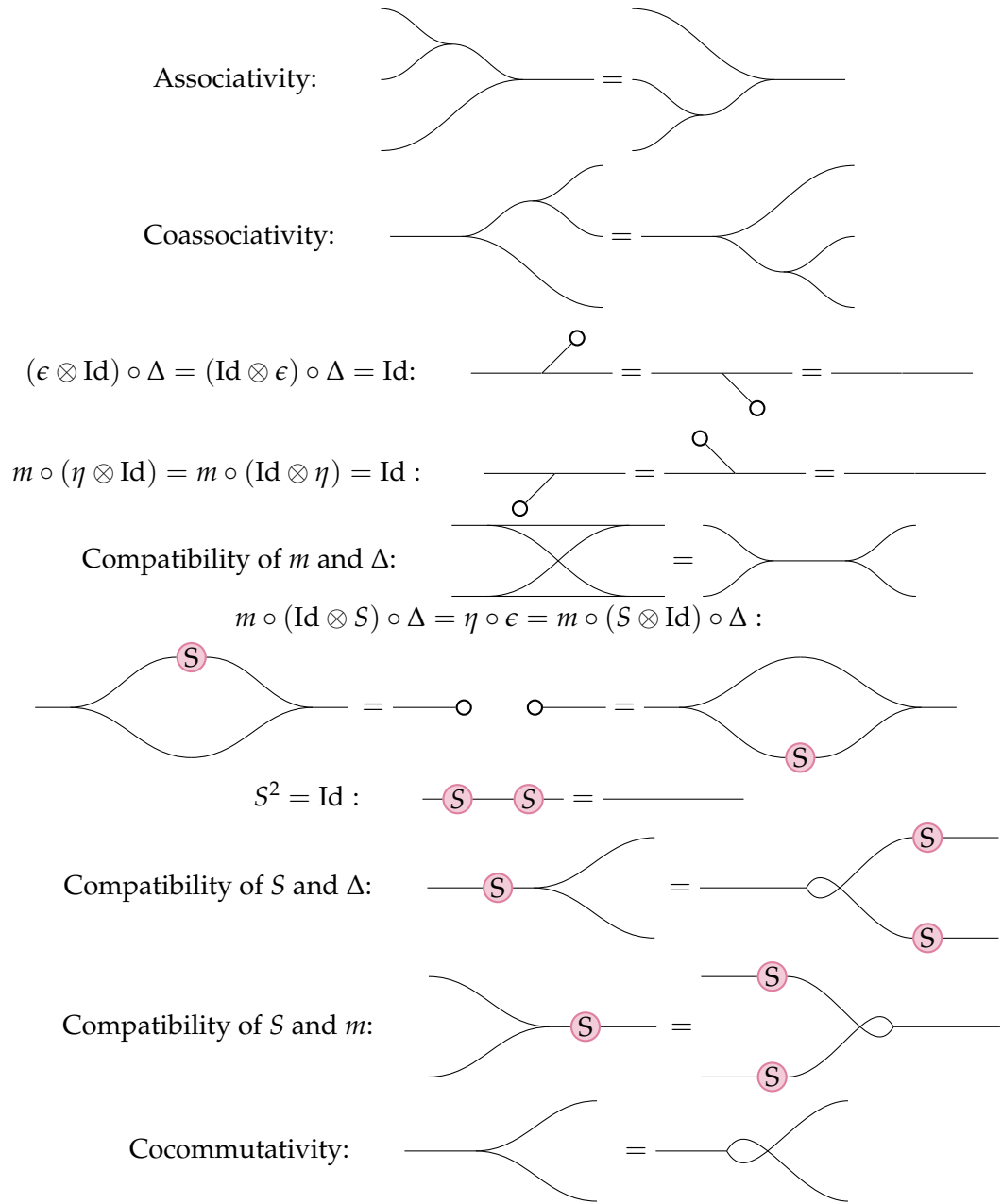
\begin{figure}
$$
\text{Associativity: \hspace{2em}}
\begin{minipage}{3cm}
\begin{tikzpicture}
\node[empty](A){};
\node[empty](B)[left of=A]{};
\node[empty](C)[left of=B]{};
\node[empty](D)[left of=C]{};
\node[empty](E)[above of=C, node distance=.5cm]{};
\node[empty](F)[below of=C, node distance=.5cm]{};
\node[empty](G)[above of=D,node distance=1cm]{};
\node[empty](H)[below  of=D, node distance =1cm]{};
\draw (A.center) to (B.center);
\draw (B.center) to[out=180,in=0] (E.center);
\draw (E.center) to[out=180,in=0] (G.center);
\draw (E.center) to[out=180,in=0] (D.center);
\draw (B.center) to[out=180,in=0] (H.center);
\end{tikzpicture}
\end{minipage}
=
\begin{minipage}{3cm}
\begin{tikzpicture}
\node[empty](A){};
\node[empty](B)[left of=A]{};
\node[empty](C)[left of=B]{};
\node[empty](D)[left of=C]{};
\node[empty](E)[above of=C, node distance=.5cm]{};
\node[empty](F)[below of=C, node distance=.5cm]{};
\node[empty](G)[above of=D,node distance=1cm]{};
\node[empty](H)[below  of=D, node distance =1cm]{};
\draw (A.center) to (B.center);
\draw (B.center) to[out=180,in=0] (F.center);
\draw (F.center) to[out=180,in=0] (H.center);
\draw (F.center) to[out=180,in=0] (D.center);
\draw (B.center) to[out=180,in=0] (G.center);
\end{tikzpicture}
\end{minipage}
$$
$$
\text{Coassociativity:\hspace{2em}}
\begin{minipage}{3cm}
\begin{tikzpicture}
\node[empty](A){};
\node[empty](B)[right of=A]{};
\node[empty](C)[right of=B]{};
\node[empty](D)[right of=C]{};
\node[empty](E)[above of=C, node distance=.5cm]{};
\node[empty](F)[below of=C, node distance=.5cm]{};
\node[empty](G)[above of=D,node distance=1cm]{};
\node[empty](H)[below  of=D, node distance =1cm]{};
\draw (A.center) to (B.center);
\draw (B.center) to[in=180,out=0] (E.center);
\draw (E.center) to[in=180,out=0] (G.center);
\draw (E.center) to[in=180,out=0] (D.center);
\draw (B.center) to[in=180,out=0] (H.center);
\end{tikzpicture}
\end{minipage}
=
\begin{minipage}{3cm}
\begin{tikzpicture}
\node[empty](A){};
\node[empty](B)[right of=A]{};
\node[empty](C)[right of=B]{};
\node[empty](D)[right of=C]{};
\node[empty](E)[above of=C, node distance=.5cm]{};
\node[empty](F)[below of=C, node distance=.5cm]{};
\node[empty](G)[above of=D,node distance=1cm]{};
\node[empty](H)[below  of=D, node distance =1cm]{};
\draw (A.center) to (B.center);
\draw (B.center) to[in=180,out=0] (F.center);
\draw (F.center) to[in=180,out=0] (H.center);
\draw (F.center) to[in=180,out=0] (D.center);
\draw (B.center) to[in=180,out=0] (G.center);
\end{tikzpicture}
\end{minipage}
$$

$$
(\epsilon\otimes \id)\circ \Delta=(\id\otimes\epsilon)\circ \Delta=\id\text{:\hspace{2em}}
\begin{minipage}{2cm}
\begin{tikzpicture}
\node[empty](aa){};
\node[empty](bb)[right of =aa]{} edge (aa.center);
\node[empty](cc)[right of =bb]{} edge (bb.center);
\node[unit](dd)[above right of =bb, node distance =.7cm]{} edge (bb.center);
\node[emptyunit][below right of =bb, node distance =.7cm]{};
\end{tikzpicture}
\end{minipage}
=
\begin{minipage}{2cm}
\begin{tikzpicture}
\node[empty](aa){};
\node[empty](bb)[right of =aa]{} edge (aa.center);
\node[empty](cc)[right of =bb]{} edge (bb.center);
\node[unit](dd)[below right of =bb, node distance =.7cm]{} edge (bb.center);
\node[emptyunit][above right of =bb, node distance =.7cm]{};
\end{tikzpicture}
\end{minipage}
=
\begin{minipage}{2cm}
\begin{tikzpicture}
\node[empty](aa){};
\node[empty](bb)[right of =aa]{} edge (aa.center);
\node[empty](cc)[right of =bb]{} edge (bb.center);
\end{tikzpicture}
\end{minipage}
$$
$$
m\circ(\eta\otimes \id)=m\circ(\id\otimes \eta)=\id:\text{\hspace{2em}}
\begin{minipage}{2cm}
\begin{tikzpicture}
\node[empty](aa){};
\node[empty](bb)[right of =aa]{} edge (aa.center);
\node[empty](cc)[right of =bb]{} edge (bb.center);
\node[unit](dd)[below left of =bb, node distance =.7cm]{} edge (bb.center);
\node[emptyunit][above left of =bb, node distance =.7cm]{};
\end{tikzpicture}
\end{minipage}
=
\begin{minipage}{2cm}
\begin{tikzpicture}
\node[empty](aa){};
\node[empty](bb)[right of =aa]{} edge (aa.center);
\node[empty](cc)[right of =bb]{} edge (bb.center);
\node[unit](dd)[above left of =bb, node distance =.7cm]{} edge (bb.center);
\node[emptyunit][below left of =bb, node distance =.7cm]{};
\end{tikzpicture}
\end{minipage}
=
\begin{minipage}{2cm}
\begin{tikzpicture}
\node[empty](aa){};
\node[empty](bb)[right of =aa]{} edge (aa.center);
\node[empty](cc)[right of =bb]{} edge (bb.center);
\end{tikzpicture}
\end{minipage}
$$
$$
\text{Compatibility of $m$ and $\Delta$:\hspace{2em}}
\begin{minipage}{3cm}
\begin{tikzpicture}
\node[empty](A){};
\node[empty](B)[right of =A,node distance =.5cm]{};
\node[empty](C)[right of =B]{};
\node[empty](D)[right of =C]{};
\node[empty](E)[right of =D, node distance = .5cm]{};
\node[empty](F)[below of =A]{};
\node[empty](G)[right of =F, node distance=.5cm]{};
\node[empty](H)[right of =G]{};
\node[empty](I)[right of =H]{};
\node[empty](J)[right of =I,node distance=.5cm]{};
\draw (A.center) to (E.center);
\draw (F.center) to (J.center);
\draw (B.center) to[in=180,out=0] (I.center);
\draw (G.center) to[in=180,out=0] (D.center);
\end{tikzpicture}
\end{minipage}
=
\begin{minipage}{3cm}
\begin{tikzpicture}
\node[empty](A){};
\node[empty](B)[right of =A, node distance=3cm]{};
\node[empty](C)[below of =A]{};
\node[empty](D)[below of =B]{};
\node[empty](E)[right of=A]{};
\node[empty](G)[below of = E, node distance=.5cm]{};
\node[empty](H)[right of = G]{};
\draw (A.center) to[out=0, in=180] (G.center);
\draw (C.center) to[out=0, in=180] (G.center);
\draw (G.center) to (H.center);
\draw (H.center) to[out=0, in=180] (B.center);
\draw (H.center) to[out=0, in=180] (D.center);
\end{tikzpicture}
\end{minipage}
$$
$$
m\circ(\id\otimes S)\circ \Delta=\eta\circ\epsilon= m\circ(S\otimes\id)\circ\Delta:
$$
$$
\begin{minipage}{4.5cm}
\begin{tikzpicture}
\node[empty](aa){};
\node[empty](bb)[above right of =aa]{};
\node[antipode](cc)[right of =bb]{S};
\draw (aa.center) to[out=0,in=180] (cc);
\node[empty](dd)[below right of =aa]{};
\node[emptyantipode](ee)[right of =dd]{};
\draw (aa.center)to[out=0,in=180] (ee.center);
\node[empty] (ff)[left of =aa, node distance=.5cm]{} edge (aa.center);
\node[empty](gg)[right of=cc]{};
\node[empty](hh)[below right of =gg]{};
\draw (cc) to[out=0,in=180] (hh.center);
\draw (ee.center) to[out=0,in=180] (hh.center);
\node[empty] [right of=hh, node distance=.5cm]{} edge (hh.center);
\end{tikzpicture}
\end{minipage}
=
\begin{minipage}{3cm}
\begin{tikzpicture}
\node[empty](aa){};
\node[unit](bb)[right of =aa]{};
\draw (aa.center) to (bb);
\node[unit](cc)[right of =bb]{};
\node[empty](dd)[right of =cc]{};
\draw (dd.center) to (cc);
\end{tikzpicture}
\end{minipage}
=
\begin{minipage}{4.5cm}
\begin{tikzpicture}
\node[empty](aa){};
\node[empty](bb)[above right of =aa]{};
\node[emptyantipode](cc)[right of =bb]{};
\draw (aa.center) to[out=0,in=180] (cc.center);
\node[empty](dd)[below right of =aa]{};
\node[antipode](ee)[right of =dd]{S};
\draw (aa.center) to[out=0,in=180] (ee);
\node[empty] (ff)[left of =aa, node distance=.5cm]{} edge (aa.center);
\node[empty](gg)[right of=cc]{};
\node[empty](hh)[below right of =gg]{};
\draw (cc.center) to[out=0,in=180] (hh.center);
\draw (ee) to[out=0,in=180] (hh.center);
\node[empty] [right of=hh, node distance=.5cm]{} edge (hh.center);
\end{tikzpicture}
\end{minipage}
$$
$$
S^2=\id:\text{\hspace{2em}}
\begin{minipage}{2cm}
\begin{tikzpicture}
\node[empty](A){};
\node[antipode](B)[right of =A, node distance=.5cm]{$S$} edge (A);
\node[antipode](C)[right of =B]{$S$} edge(B);
\node[empty](D)[right of =C, node distance=.5cm]{} edge (C);
\end{tikzpicture}
\end{minipage}
=
\begin{minipage}{2cm}
\begin{tikzpicture}
\node[empty](A){};
\node[empty](B)[right of =A, node distance=.5cm]{};
\node[emptyantipode](C)[right of =B]{};
\node[empty](D)[right of =C, node distance=.5cm]{} edge (A);
\end{tikzpicture}
\end{minipage}
$$
$$
\text{Compatibility of $S$ and $\Delta$:\hspace{2em}}
\begin{minipage}{3.3cm}
\begin{tikzpicture}
\node[empty](A){};
\node[antipode](B)[right of=A]{S} edge (A);
\node[empty](C)[right of =B, node distance=.5cm]{};
\node[empty](D)[above right of=C]{};
\node[empty](E)[right of = D]{};
\node[empty](F)[below right of=C]{};
\node[empty](G)[right of =F]{};
\draw (B) to[in=180,out=0] (C.center);
\draw (C.center) to[in=180,out=0] (E.center);
\draw (C.center) to[in=180,out=0] (G.center);
\end{tikzpicture}
\end{minipage}
=
\begin{minipage}{3.3cm}
\begin{tikzpicture}
\node[empty](A){};
\node[empty](B)[right of=A]{};
\node[empty](C)[right of =B, node distance=.5cm]{};
\draw (A.center) to (C.center);
\node[empty](D)[above right of=C]{};
\node[antipode](E)[right of = D]{S};
\node[empty](F)[below right of=C]{};
\node[antipode](G)[right of =F]{S};
\draw (C.center) to[in=180,out=-60] (E);
\draw (C.center) to[in=180,out=60] (G);
\node[empty](H)[right of=E]{};
\draw (E) to (H.center);
\node[empty](I)[right of=G]{};
\draw (G) to (I.center);
\end{tikzpicture}
\end{minipage}
$$
$$
\text{Compatibility of $S$ and $m$:\hspace{2em}}
\begin{minipage}{3.3cm}
\begin{tikzpicture}
\node[empty](A){};
\node[antipode](B)[left of=A]{S};
\draw (B) to (A.center);
\node[empty](C)[left of =B, node distance=.5cm]{};
\draw (B) to (C.center);
\node[empty](D)[above left of=C]{};
\node[empty](E)[left of = D]{};
\node[empty](F)[below left of=C]{};
\node[empty](G)[left of =F]{};
\draw (C.center) to[in=0,out=180] (E.center);
\draw (C.center) to[in=0,out=180] (G.center);
\end{tikzpicture}
\end{minipage}
=
\begin{minipage}{3.3cm}
\begin{tikzpicture}
\node[empty](A){};
\node[empty](B)[left of=A]{};
\node[empty](C)[left of =B, node distance=.5cm]{};
\draw (C.center) to (A.center);
\node[empty](D)[above left of=C]{};
\node[antipode](E)[left of = D]{S};
\node[empty](F)[below left of=C]{};
\node[antipode](G)[left of =F]{S};
\draw (C.center) to[in=0,out=-120] (E);
\draw (C.center) to[in=0,out=120] (G);
\node[empty](H)[left of=E]{};
\draw (H.center) to (E);
\node[empty](I)[left of=G]{};
\draw (I.center) to (G);
\end{tikzpicture}
\end{minipage}
$$
$$
\text{Cocommutativity:\hspace{2em}}
\begin{minipage}{3cm}
\begin{tikzpicture}
\node[empty](aa){};
\node[empty](bb)[right of=aa]{};
\draw (bb.center) to (aa.center);
\node[empty](cc)[above right of =bb]{};
\node[empty](ee)[right of=cc]{};
\node[empty](dd)[below right of =bb]{};
\node[empty](ff)[right of=dd]{};
\draw (bb.center) to[out=0,in=180] (ee.center);
\draw (bb.center) to[out=0,in=180] (ff.center);
\end{tikzpicture}
\end{minipage}
=
\begin{minipage}{3cm}
\begin{tikzpicture}
\node[empty](aa){};
\node[empty](bb)[right of=aa]{};
\draw (bb.center) to (aa.center);
\node[empty](cc)[above right of =bb]{};
\node[empty](ee)[right of=cc]{};
\node[empty](dd)[below right of =bb]{};
\node[empty](ff)[right of=dd]{};
\draw (bb.center) to[out=-60,in=180] (ee.center);
\draw (bb.center) to[out=60,in=180] (ff.center);
\end{tikzpicture}
\end{minipage}
$$
\caption{Graphical depictions of the defining relations for cocommutative Hopf algebras.}
\label{graph-calc}
\end{figure}

By iteratively composing product and coproduct we define maps $m^{k}\colon H^{\otimes (k+1)}\to H$ and $\Delta^{k}\colon H\to H^{\otimes (k+1)}$.
By convention $m^{-1}=\eta$ and $\Delta^{-1}=\epsilon$. We will often use a type of Sweedler notation for the coproduct: $\Delta^{k-1}(h)=\sum h_{(1)}\otimes \cdots\otimes h_{(k)}$, often dropping the summation sign to just  $h_{(1)}\otimes \cdots\otimes h_{(k)}$.


Here are some examples of cocommutative Hopf algebras to keep in mind.
\begin{examples}\
\begin{enumerate}
\item If $G$ is a group, the group algebra $\F[G]$ is a Hopf algebra with the following operations defined on group elements and extended linearly:
\begin{center}
$m(g_1\otimes g_2)=g_1g_2$,\,\,
$\Delta(g)=g\otimes g$,\,\,
$S(g)=g^{-1}$,\,\,
$\epsilon (g)=1\in\F$.
\end{center}

\item If $\mathfrak g$ is a Lie algebra, then the universal enveloping algebra $U(\mathfrak g)$ is defined as the quotient of the free associative algebra $T(\mathfrak g)$ generated by $\mathfrak g$ by the relations $[X,Y]=XY-YX$. The operations are defined on products of Lie algebra elements $X_I=X_{i_1}\ldots X_{i_n}$ as follows:\\
$m(X_I\otimes Y_J)=X_IY_J$,\,\,
$\Delta(X_I)=\sum X_{I'}\otimes X_{I''}$ where the sum is over all pairs of index sets that shuffle together to give $I$,\,\,
 $S(X_1\ldots X_n)=(-1)^{n} X_n\ldots X_1$,\,\,
$\epsilon (X_I)=0$ on nontrivial products of primitive elements, and $\epsilon(1)=1$.
\item The two extreme cases of universal enveloping algebras are as follows.
\begin{enumerate}
\item If $\mathfrak g=\sL(V)$ is the free Lie algebra generated by $V$, then $U(\mathfrak g)=T(V)$.
\item If $\mathfrak g$ is abelian, then $U(\mathfrak g)$ is the symmetric algebra $S(\mathfrak g)$.
\end{enumerate}
\item
Finally, we can combine the first two cases. Let $\mathfrak{g}$ be a Lie algebra and $G$  a group acting on $\mathfrak{g}$ by automorphisms. We can form a semidirect product%
\footnote{Sometimes this is called the smash product.}
$\F[G] \ltimes U(\mathfrak{g})$ which as a set is just the tensor product but the multiplication is twisted by the action. A result 
known as the Cartier-Kostant-Milnor-Moore Theorem (for example, as stated in~\cite[Thm 1.1]{And} )
asserts that for an algebraically closed field of characteristic $0$, every cocommutative Hopf algebra is of this form,
see~\cite[sec 7]{MM}, \cite[Thm 2]{Ca}, \cite[Th. 8.1.5,13.0.1]{Sw}, \cite[Th 3.3]{Ko}.
\end{enumerate}
\end{examples}

An element $h\in H$ is said to be \emph{primitive} if $\Delta(h) = h \otimes 1 + 1 \otimes h$.
The Hopf algebra axioms imply that $\epsilon(h)=0$ and $S(h) = -h$ for any primitive element $h$.
Also the commutator $[h_1,h_2]=h_1h_2-h_2h_1$ is a primitive element if both $h_1$ and $h_2$ are primitive, i.e.,
the set of primitive elements forms a Lie algebra.

A Hopf algebra $H$ is said to be of \emph{primitively generated} 
if $H$ is generated as an algebra by primitive elements. It is clear that any primitively generated Hopf algebra  is cocommutative.
As a special case of the Cartier-Kostant-Milnor-Moore theorem, such a Hopf algebra can be identified with $U(\mathfrak{g})$, where $\mathfrak{g}$ is the set of primitive elements.


\subsection{The PBW isomorphism} 


For any Lie algebra $\mathfrak{g}$ the universal enveloping algebra $U(\mathfrak g)$ has a filtration
$$
U_0 \subset U_1 \subset U_2 \subset \cdots \subset U_n \subset \cdots
$$
where $U_0 = \F$, $U_1 = \F \oplus \mathfrak{g}$ and $U_n$ is the span of all products of at most $n$ elements from $\mathfrak{g}$.
Notice that this filtration is preserved by $m$, $\Delta$ and $S$, i.e.,
$$
m(U_i\otimes U_j) \subset U_{i+j},
\quad
\Delta(U_i) \subset  \sum_{p+q=i} U_p \otimes U_{q},
\quad \mbox{and} \quad
S(U_i) \subset U_i.
$$

The well-known PBW theorem gives that $U_{i}/U_{i-1}$ is isomorphic to $\Sym^i(\mathfrak{g})$ -- the map
$\sigma_i :  \Sym^i(\mathfrak{g}) \to U_{i}$ given by
$$
\sigma_i(X_1\ldots X_i)=\frac{1}{|\sym{i}|}\sum_{\alpha\in\sym{i}} X_{\alpha(1)}\ldots X_{\alpha(i)}
$$
induces the isomorphism between the vector spaces $\Sym^i(\mathfrak{g})$ and $U_{i}/U_{i-1}$.
The inverse of the map $\sigma_i$ is not difficult to write down explicitly: for any $h \in U_i$
$$
(\id - \eta\epsilon)^{\otimes i} \circ \Delta^{i-1} (h) \in U^{\otimes i}.
$$
In fact the image lies inside $U_1^{\otimes i} \subset U^{\otimes i}$
because  $\Delta^{i-1} (h) \in \sum_{p_1 + \cdots + p_i =i} U_{p_1} \otimes \cdots \otimes U_{p_i}$
and  $(\id - \eta\epsilon)^{\otimes i}$ is zero on the terms in the above sum which contain $U_0$.
Using the projection $U_1 \twoheadrightarrow \mathfrak{g}$ this becomes an element of $\mathfrak{g}^{\otimes i}$. Since the
comultiplication is cocommutative this is a symmetric tensor and can be viewed as element of $\Sym^i(\mathfrak{g})$.
This leads to a map $\pi_i:U_i \to \Sym^i(\mathfrak{g})$ with kernel $U_{i-1}$ which is the inverse of $\sigma_i$ (up to normalization constant).
It is important to note that the above formula is only valid for elements in $U_i$ and not for arbitrary elements in
$U(\mathfrak g)$.

The maps $\{\sigma_i\}_{i \geq 0}$ can be combined in a map $\sigma\colon \Sym(\mathfrak g)\to U(\mathfrak g)$
which becomes a vector space isomorphism. A direct computation shows that $\sigma$ preserves the comultuplication
$\Delta$ and is an isomorphism of coalgebras.

The following claim is well known:
\begin{claim}
\label{cl:powers}
The image of $\sigma_i$ is generated as vector space by elements $\{X^i = \sigma_i(X^i)\}$ for $X\in \mathfrak{g}$.
\end{claim}

\begin{lemma}
\label{lm:class2}
If the Lie algebra satisfies the equation $[[X,Y],X] = [[X,Y],Y]=0$, in particular if $\mathfrak{g}$ is nilpotent of class $2$ then there exists constants $c_{n,k,i}$ and $d_{n,k,i}$ such that
$$
X^nY^k=\sigma_n(X^n)Y^k = \sigma_{n+k}(X^nY^k) + \sum_{i >0} c_{n,k,i} \sigma_{n+k-i}(X^{n-i} Y^{k-i}[X,Y]^i)
$$
and
$$
Y^kX^n=Y^k\sigma_n(X^n) = \sigma_{n+k}(X^nY^k) + \sum_{i >0} d_{n,k,i} \sigma_{n+k-i}(X^{n-i} Y^{k-i}[X,Y]^i).
$$
\end{lemma}
\begin{proof}
The subalgebra generated by $X$ and $Y$ is $3$ dimensional Heisenberg Lie algebra $\mathfrak{h} =\mathrm{span}\{X,Y,[X,Y]\}$ and it is enough to do all computations in $U(\mathfrak{h})$. This algebra is multigraded and the elements $\sigma_{n+k-i}(X^{n-i} Y^{k-i}[X,Y]^i)$ form a basis of the graded component of multidegree  $(n,k)$, because $X^{n-i} Y^{k-i}[X,Y]^i$ form a basis of the graded component of $\Sym(\mathfrak{h})$.
Therefore $X^nY^k$ and $Y^kX^n$ can be expressed as linear combinations of the above elements.
Thus, we only need to show that the coefficient in front of $\sigma_{n+k}(X^nY^k)$ is $1$ which follows
by abelianizing.
\end{proof}

\section{The operad $H\cO$}
\label{sec:Hopf_operad}
Suppose $H$ is a cocommutative Hopf algebra and $\cO$ is an operad with unit (in the category of $\F$-vector spaces). (For background on operads see \cite{operad}.) We let $\cO[n]$ denote the vector space spanned by operad elements with $n$ inputs and one output, $n$ being referred to as the \emph{arity}. If $\cO$ is cyclic, we let $\cO\arity{n}=\cO[n-1]$ as an $\sym{n}$-module, where $\sym{n}$ denotes the symmetric group.

Regard $H$ as an operad with elements only of arity $1$ and operad composition given by algebra multiplication. The antipode $S$ turns $H$ into a cyclic operad: the $\sym{2}$ action sends $h$ to $S(h)$.%
\footnote{Notice that we do not need the comultiplication on $H$ to turn it into a cyclic operad.}

\begin{definition}\
\begin{enumerate}
\item Let $\cO_1$ and $\cO_2$ be operads with unit. Define $\cO_1*\cO_2$ to be the operad \emph{freely generated by $\cO_1$ and $\cO_2$}. This is defined to be the operad consisting of trees with vertices of valence $\geq 2$ labeled by elements of $\cO_1$ or $\cO_2$. Composing two elements of $\cO_i$ for $i=1,2$ along a tree edge is considered the same element of $\cO_1*\cO_2$, and the units of $\cO_1$ and $\cO_2$ are identified and equal to the unit of $\cO_1*\cO_2$.
\item Let $H\cO$ be the quotient of $H*\cO$ by the relation that $h$ commutes with an element of $\cO$ via the comultiplication map as in the figure below. (The fact that $1_\cO=1_H$ is also included for emphasis.)
\end{enumerate}
\end{definition}

\begin{tikzpicture}
\node[empty](aa){};
\node[operad](bb)[left of=aa]{$1_\cO$} edge (aa);
\node[empty](cc)[left of=bb]{}  edge (bb);

\node[empty](dd)[left of=cc]{};
\node[hopf](ee)[left of=dd]{$\,\,1_H$}  edge (dd);
\node[empty](ff)[left of=ee]{}  edge (ee);

\path(cc.center) to node[anchor=mid]{$=$} (dd.center);

\node[break](br)[right of=aa]{};

\node[empty](a)[right of=br]{};
\node[hopf](b)[right of=a]{$\,\,h$} edge (a);
\node[operad](c)[right of=b, node distance=1.3cm]{$o$} edge (b);
\node[empty](d)[right of=c]{} edge (c);
\node[empty](e)[above right of=c]{} edge (c);
\node[empty](f)[below right of=c]{} edge (c);

\node[empty](a')[right of=d]{};
\node[operad](b')[right of=a']{$o$} edge (a');
\node[hopf](c')[right of=b', node distance=1.8cm]{$\,\,h_{(2)}$} edge (b');
\node[empty](d')[right of=c']{} edge (c');
\node[hopf](e')[above of=c']{$\,\,h_{(1)}$}; \draw (b') to[out=45,in=180] (e');
\node[empty](f')[right of=e']{} edge (e');
\node[hopf](g')[below of=c']{$\,\,h_{(3)}$}; \draw (b') to[out=-45,in=180] (g');
\node[empty](h')[right of=g']{} edge (g');

\path(d.center) to node[anchor=mid]{$=$} (a'.center);

\end{tikzpicture}

Note the use of Sweedler notation hides the fact that the coproduct is actually a sum of pure tensors.
The co-commutativity of the Hopf algebra $H$ easily implies

\begin{claim}
The arity $n$ part of the operad $H\cO$ is
$$
H\cO[n] =\cO[n] \otimes H^{\otimes n},
$$
where the symmetric group $\sym{n}$ acts on both tensor factors simultaneously.
\end{claim}
\begin{proof}
Using that the elements of $H$ and the operad elements ``commute'' one can express any element from $H\cO$ as an element of the operad $\cO$ followed by several elements from $H$ as in the following example
\begin{center}
\begin{tikzpicture}

\node[empty](a)[]{};
\node[hopf](b)[right of=a]{$\,h\,$} edge (a);
\node[operad](c)[right of=b, node distance=1.3cm]{$o$} edge (b);
\node[hopf](d)[right of=c,node distance=1.3cm]{$\,h'\,$} edge (c);
\node[empty](e)[below right of=c]{} edge (c);
\node[operad](f)[right of=d, node distance=1.3cm]{$o'$} edge (d);
\node[empty](g)[right of=f]{} edge (f);
\node[empty](h)[above right of=f]{} edge (f);

\node[empty](a')[right of=g]{};
\node[operad](b')[right of=a']{$\,\,o \circ_{1} o'\,\,$} edge (a');
\node[hopf](c')[right of=b', node distance=2.5cm]{$\,\,h_{(2)}h_{(2)}'$} edge (b');
\node[empty](d')[right of=c', node distance=1.3cm]{} edge (c');
\node[hopf](e')[above of=c', node distance=1.2cm]{$\,\,h_{(1)}h'_{(1)}$}; \draw (b') to[out=45,in=180] (e');
\node[empty](f')[right of=e', node distance=1.3cm]{} edge (e');
\node[hopf](g')[below of=c', node distance=1.2cm]{$\,\,\,\,\,h_{(3)}\,\,\,\,\,$}; \draw (b') to[out=-45,in=180] (g');
\node[empty](h')[right of=g', node distance=1.3cm]{} edge (g');

\path(g) to node{$=$} (a');

\end{tikzpicture}
\end{center}
Guided by the above construction, one can easily define an operad structure on the $S$-module with arity $n$ part $\cO[n]\otimes H^{\otimes n}$. That means we have a surjective morphism of operads $\cO[n]\otimes H^{\otimes n}\twoheadrightarrow H\cO[n]$. By freeness this map must be an isomorphism.
\end{proof}

\begin{remark}
If the algebra $H$ is not cocommutative then this action will not be well defined. A similar construction can be made for any Hopf algebra and non-symmetric operad $\cO$.
\end{remark}

As we mentioned earlier, the antipode gives rise to a cyclic operad structure on $H\subset H\cO$. If $\cO$ is also a cyclic operad, then we get an induced cyclic operad structure on $H\cO$.

%
%
%

It is worthwhile to spell this construction out in a little more detail for $H\Lie$ as follows. As before, we regard the $H$ part of the operad as being a two-valent vertex with one input leaf, one output leaf and $H$-labeled inside.
Then $H\Lie\arity{n}$ is generated by  trees with vertices of valence $\leq 3$, with an ordering of the edges at each bivalent vertex, a cyclic ordering of the edges at the  trivalent vertices, and the bivalent vertexes are labeled with elements of the Hopf algebra $H$. These trees satisfy the following relations:

\begin{enumerate}
\item (Multiplication)
\begin{tikzpicture}[baseline=(eq.south)]
\node[empty](a){};
\node[hopf](b)[right of=a]{$\,\,h_1$} edge (a);
\node[hopf](c)[right of=b,node distance=1.9cm]{$\,\,h_2$} edge (b);
\node[empty](d)[right of=c]{} edge (c);

\node[empty](e)[right of=d]{};
\node[hopf](f)[right of=e]{$\,\,h_1h_2$} edge (e);
\node[empty](g)[right of=f,node distance=1.3cm]{} edge (f);

\path (d.center) to node[anchor=mid](eq){$=$} (e.center);
\end{tikzpicture}

\item (Antipode)
\begin{tikzpicture}[baseline=(eq.south)]
\node[empty](a){};
\node[hopf](b)[right of=a]{$\,\,h\,$} edge (a);
\node[empty](c)[right of=b]{} edge (b);

\node[empty](e)[right of=c]{};
\node[hopf,](f)[right of=e,node distance=2.5cm]{$\,\,S(h)$};
\node[empty](fup)[above of=f,node distance=.5cm]{};
\node[empty](fdown)[below of=f,node distance=.3cm]{};
\node[empty](fldown)[left of=fdown,node distance=1.0cm]{};
\node[empty](g)[right of=f,node distance=2cm]{};
\draw (e) to[out=0,in=180] (fup) to[out=0,in=20] (f.east);
\draw (f.west) to[out=190,in=150] (fldown) to[out=330,in=180] (g);

\path (c) to node(eq){$=$} (e);
\end{tikzpicture}

\item (Removal of the Identity)
\begin{tikzpicture}[baseline=(eq.south)]
\node[empty](a){};
\node[hopf](b)[right of=a]{$\,\,1_H$} edge (a);
\node[empty](c)[right of=b,node distance=1.3cm]{} edge (b);

\node[empty](e)[right of=c]{};
\node[empty](g)[right of=e,node distance=2cm]{} edge (e);

\path (c) to node(eq){$=$} (e);
\end{tikzpicture}

\item (Comultiplication)
\begin{tikzpicture}[baseline=(eq.south)]
\node[empty](a){};
\node[hopf](b)[right of=a,node distance=1.3cm]{$\,\,h$} edge (a);
\node[empty](c)[right of=b]{};
\draw (c.center) to (b);
\node[empty](d)[right of=c]{};
\node[empty](dd)[right of=d,node distance=0.1cm]{};
\node[empty](e)[above of=dd,node distance=0.4cm]{};
\node[empty](f)[below of=dd,node distance=0.4cm]{};
\draw (c.center) to[out=45,in=180] (e.center);
\draw (c.center) to[out=-45,in=180] (f.center);

\node[empty](a')[right of=dd]{};
\node[empty](b')[right of=a',node distance=0.6cm]{};
\draw (a'.center) to (b'.center);
\node[empty](c')[right of=b', node distance=1.3cm]{};
\node[hopf](e')[above of=c',node distance=0.6cm]{$\,\,h_{(1)}$};
\draw (b'.center) to[out=45,in=180] (e');
\node[empty](f')[right of=e', node distance=1.3cm]{} edge (e');
\node[hopf](g')[below of=c',node distance=0.6cm]{$\,\,h_{(2)}$};
\draw (b'.center) to[out=-45,in=180] (g');
\node[empty](h')[right of=g',node distance=1.3cm]{} edge (g');

\path(dd) to node(eq){$=$} (a');
\end{tikzpicture}

\item (IHX)
$\begin{minipage}{1.732cm}
\begin{tikzpicture} [rotate=90]
\draw (0,0) -- (-.35,.6062);
\draw[densely dotted] (-.35,.6062)--(-.5,.866);
\draw (0,0) -- (-.35,-.6062);
\draw[densely dotted] (-.35,-.6062)-- (-.5,-.866);
\draw (0,0) -- (1,0);
\draw (1,0) -- (1.35,.6062);
\draw[densely dotted] (1.35,.6062)--(1.5,.866);
\draw (1,0) -- (1.35,-.6062);
\draw[densely dotted] (1.35,-.6062)--(1.5,-.866);
\end{tikzpicture}
\end{minipage}
=
\begin{minipage}{2cm}
\begin{tikzpicture}
\draw (0,0) -- (-.35,.6062);
\draw[densely dotted] (-.35,.6062)--(-.5,.866);
\draw (0,0) -- (-.35,-.6062);
\draw[densely dotted] (-.35,-.6062)-- (-.5,-.866);
\draw (0,0) -- (1,0);
\draw (1,0) -- (1.35,.6062);
\draw[densely dotted] (1.35,.6062)--(1.5,.866);
\draw (1,0) -- (1.35,-.6062);
\draw[densely dotted] (1.35,-.6062)--(1.5,-.866);
\end{tikzpicture}
\end{minipage}
-
\begin{minipage}{1.732cm}
\begin{tikzpicture}
\draw (0,0) -- (.4,.231);
\draw (.6,.346) -- (1.35,.78);
\draw[densely dotted] (1.35,.78) -- (1.5,.866);
\draw (0,0) -- (-.35,-.6062);
\draw[densely dotted] (-.35,-.6062)-- (-.5,-.866);
\draw (0,0) -- (1,0);
\draw (1,0) -- (-.35,.78);
\draw[densely dotted] (-.35,.78)--(-.5,.866);
\draw (1,0) -- (1.35,-.6062);
\draw[densely dotted] (1.35,-.6062)--(1.5,-.866);
\end{tikzpicture}
\end{minipage}
$

\item (AntiSymmetry)
\begin{tikzpicture}[baseline=(eq.south)]
\node[empty](a){};
\node[empty](c)[right of=a,node distance=0.6cm]{};
\draw (c.center) to (a.center);
\node[empty](d)[right of=c]{};
\node[empty](e)[above right of=c]{};
\draw (c.center) to[out=40,in=180] (e.center);
\node[empty](f)[below right of=c]{};
\draw (c.center) to[out=-40,in=180] (f.center);

\node[empty](a')[right of=d]{};
\node[empty](b')[right of=a',node distance=0.6cm]{};
\draw (b'.center) to (a'.center);
\node[empty](c')[right of=b', node distance=0.3cm]{};
\node[empty](cup')[above of=c',node distance=0.3cm]{};
\node[empty](cdown')[below of=c',node distance=0.3cm]{};
\node[empty](d')[right of=c', node distance=1.5cm]{};
\node[empty](dup')[above of=d',node distance=0.7cm]{};
\node[empty](ddown')[below of=d',node distance=0.7cm]{};
\draw (b'.center) to[out=60,in=135] (ddown'.center);
\draw (b'.center) to[out=-65,in=-135] (dup'.center);

\node(zero)[right of=d']{$0$};

\path(d) to node{$+$} (a');

\path(d') to node(eq){$=$} (zero);
\end{tikzpicture}
\end{enumerate}

%
%

Given a vector space $V$, define a cyclic operad $\Lie_V\arity{n}$ which is spanned by elements of $\Lie$ with $n$ numbered univalent vertices representing the input/output slots of the operad, and the rest of the univalent vertices labeled by elements of $V$.
Here we have a choice how to define the spaces  $\Lie_V\arity{0}$ and $\Lie_V\arity{1}$\footnote{Notice that the arity $n$ part of the $\cO$ when considered as cyclic operad is denoted by
{$\cO\arity{n+1}$}.
}.
The above definition gives that these spaces are not empty:
\begin{enumerate}
\item $\Lie_V\arity{1}\cong\sL(V)$ is the free Lie algebra generated by $V$.
\item  If $V$ is symplectic, then $\Lie_V\arity{0}$ (as a vector space) is Kontsevich's Lie algebra for the operad $\Lie$ with coefficients in $V$.
\end{enumerate}
Another possible choice leads to the reduced operad $\overline{\Lie_V}$ where the spaces $\Lie_V\arity{0}$ and $\Lie_V\arity{1}$ are empty.\footnote{Since the arity $-1$ part of a cyclic operad is not involved in any compositions, the exact definition of $\Lie_V\arity{0}$ is irrelevant. However this is not the case for arity $0$ part.}

Notice that arity $1$ part is also highly nontrivial:  $\Lie_V\arity{2} \cong T(V)$ as can be seen from the following picture.  The action of $\sym{2}$ is as the antipode on $T(V)$.
\begin{center}
\begin{tikzpicture}

\node[empty](a){};
\node[empty](b)[right of=a]{};
\draw (a.center) to (b.center);
\node[empty](b')[below right of=b]{$v_1$} edge (b.center);
\node[empty](c)[right of=b]{};
\draw (b.center) to (c.center);
\node[empty](c')[below right of=c]{$v_2$} edge (c.center);
\node[empty](d)[right of=c]{};
\draw (c.center) to (d.center);
\node[empty](d')[below right of=d]{$v_3$} edge (d.center);
\node[empty](e)[right of=d]{};
\draw (d.center) to (e.center);
\node[empty](e')[below right of=e]{$v_4$} edge (e.center);
\node[empty](f)[right of=e]{};
\draw (e.center) to (f.center);

\node[empty](x)[right of=f,node distance=2cm]{$v_1v_2v_3v_4$};

\path(f) to node{$\rightsquigarrow$} (x);
\end{tikzpicture}
\end{center}

This observation gives the following

\begin{proposition}
\label{prop:Lie-iso}
Let $H=T(V)$ be the tensor (Hopf) algebra. Then we have $H\Lie=  \overline{\Lie_V}$.
\end{proposition}
\begin{proof}
Given an element of $\overline{\Lie_V}$, we can think of it as a tree with numbered leaves with  $V$-labeled trees growing off of it. (We need $\overline{\Lie_V}\arity{0}=\overline{\Lie_V}\arity{1}=0$ for this to work!)
 The IHX relation allows us to replace all the hanging trees with hairs, i.e., trees without internal vertices.
Thus we can decompose this element as a composition of elements of $\Lie$ and $\Lie_V\arity{2}\cong T(V)$. The comultiplication axiom for $H\Lie$ when applied to a primitive element $v\in H$ being pushed past a trivalent vertex is as follows:
\begin{center}
\begin{tikzpicture}[baseline=(eq.south)]
\node[empty](a){};
\node[hopf](b)[right of=a,node distance=0.6cm]{$\,\,v$} edge (a.center);
\node[empty](c)[right of=b]{};
\draw (b) to (c.center);
\node[empty](d)[right of=c]{};
\node[empty](e)[above right of=c]{};
\draw (e.center) to (c.center);
\node[empty](f)[below right of=c]{};
\draw (f.center) to (c.center);

\node[empty](a')[right of=d]{};
\node[empty](b')[right of=a',node distance=0.6cm]{};
\draw(a'.center) to (b'.center);
\node[empty](c')[right of=b', node distance=1.3cm]{};
\node[hopf](e')[above of=c',node distance=0.6cm]{$\,\,v$};
\draw (b'.center) to[out=45,in=180] (e');
\node[empty](f')[right of=e', node distance=0.9cm]{};
\draw (f'.center) to (e');
\node[hopf](g')[below of=c',node distance=0.6cm]{$\,\,1$};
\draw (b'.center) to[out=-45,in=180] (g');
\node[empty](h')[right of=g',node distance=0.9cm]{};
\draw (h'.center) to (g');

\node[empty](a'')[right of=a', node distance=3.8cm]{};
\node[empty](b'')[right of=a'',node distance=0.6cm]{};
\draw (b''.center) to (a''.center);
\node[empty](c'')[right of=b'', node distance=1.3cm]{};
\node[hopf](e'')[above of=c'',node distance=0.6cm]{$\,\,1$};
\draw (b''.center) to[out=45,in=180] (e'');
\node[empty](f'')[right of=e'', node distance=0.9cm]{};
\draw (f''.center) to (e'');
\node[hopf](g'')[below of=c'',node distance=0.6cm]{$\,\,v$};
\draw (b''.center) to[out=-45,in=180] (g'');
\node[empty](h'')[right of=g'',node distance=0.9cm]{};
\draw (h''.center) to (g'');

\node[empty][left of =a'', node distance=.5cm]{$+$};
\path(d) to node(eq){$=$} (a');

\end{tikzpicture}
\end{center}
This is exactly an IHX relation involving a $v$-labeled hair:
\begin{center}
\begin{tikzpicture}[baseline=(eq.south)]
\node[empty](a){};
\node[empty](b)[right of=a,node distance=0.6cm]{};
\draw (b.center) edge (a.center);
\node[below of=b, node distance=.7cm]{$v$} edge (b.center);
\node[empty](c)[right of=b]{};
\draw (c.center) to (b.center);
\node[empty](d)[right of=c]{};
\node[empty](e)[above right of=c]{};
\draw (e.center) to (c.center);
\node[empty](f)[below right of=c]{};
\draw (f.center) to (c.center);

\node[empty](a')[right of=d]{};
\node[empty](b')[right of=a',node distance=0.6cm]{};
\draw (b'.center) edge (a'.center);
\node[empty](c')[right of=b', node distance=1.3cm]{};
\node[empty](e')[above of=c',node distance=0.6cm]{};
\draw (b'.center) to[out=45,in=180] (e'.center);
\node[below of =e', node distance=.7cm]{$v$} edge (e'.center);
\node[empty](f')[right of=e', node distance=0.9cm]{};
\draw (f'.center) to (e'.center);
\node[empty](g')[below of=c',node distance=0.6cm]{};
\draw (b'.center) to[out=-45,in=180] (g'.center);
\node[empty](h')[right of=g',node distance=0.9cm]{};
\draw (h'.center) to (g'.center);

\node[empty](a'')[right of=a', node distance=3.8cm]{};
\node[empty](b'')[right of=a'',node distance=0.6cm]{};
\draw (b''.center) edge (a''.center);
\node[empty](c'')[right of=b'', node distance=1.3cm]{};
\node[empty](e'')[above of=c'',node distance=0.6cm]{};
\draw (b''.center) to[out=45,in=180] (e''.center);
\node[empty](f'')[right of=e'', node distance=0.9cm]{};
\draw (f''.center) to (e''.center);
\node[empty](g'')[below of=c'',node distance=0.6cm]{};
\draw (b''.center) to[out=-45,in=180] (g''.center);
\node[below of=g'', node distance=.7cm]{$v$} edge(g''.center);
\node[empty](h'')[right of=g'',node distance=0.9cm]{};
\draw (h''.center) to (g''.center);

\node[empty][left of =a'', node distance=.5cm]{$+$};
\path(d) to node(eq){$=$} (a');

\end{tikzpicture}
\end{center}
\end{proof}

\section{Hopf algebras and groups}
\label{sec:action_on_Hn}

As before, let $H$ be a cocommutative Hopf algebra. Our goal in this section is to show that there is a natural $\Aut(F_n)$ action on $H^{\otimes n}$ and to introduce an appropriate quotient $\overline{H^{\otimes n}}$ on which $\Out(F_n)$ acts.

\begin{definition}\label{def:homh}
Given a group $G$, and a cocommutative Hopf algebra $H$, define $\HomH(G;H)$ as the set of sequences of functions $\varphi=\{\varphi_k\}$, $\varphi_k\colon G^k\to H^{\otimes k}$ satisfying the following properties
\begin{enumerate}
\item $\varphi_k$ commutes with the action of $\sym k$.
\item $\varphi_k(1,g_1,\ldots,g_{k-1})= \eta(1)\otimes \varphi_{k-1}(g_1,\ldots,g_{k-1})$
\item $(\epsilon\otimes \id^{k-1})\varphi_k(g_1,\ldots,g_k)=\varphi_{k-1}(g_2,\ldots,g_k)$
\item $\varphi_k(g^{-1},g_1,\ldots,g_{k-1})=(S\otimes \id^{k-1})\varphi_k(g,g_1,\ldots g_k)$
\item $\varphi_k(g,g,g_1,\ldots,g_{k-2})=(\Delta\otimes \id^{k-2})\varphi_{k-1}(g,g_1,\ldots,g_{k-2})$
\item $\varphi_k(g_1g_2,g_3,\ldots,g_{k+1})=(m\otimes \id^{k})\varphi_{k+1}(g_1,g_2,\ldots,g_{k+1})$
\end{enumerate}
\end{definition}

\begin{remark}
One can view the elements of the set $\HomH(G;H)$ as natural transformations as follows: Let $\mathcal{HOPF}$ be the PROP  of Hopf algebras, 
i.e., this is a monoidal category with the objects the set of natural numbers and morphisms generated by the ``multiplication'' $m \in \Hom(2,1)$, ``comultiplication" $\Delta \in \Hom(1,2)$, ``unit" $\epsilon \in \Hom(0,1)$, ``counit" $\eta \in \Hom(1,0)$, ``antipode" $S\in \Hom(1,1)$ and the ``flip" $\tau \in \Hom(2,2)$, which satisfy all the axioms of the Hopf algebras. For any Hopf algebra $H$ there is a canonical monoidal functor $\mathfrak{F}_H : \mathcal{HOPF} \to \End_{\F}(H)$ which sends the object $n$ to $H^{\otimes n}$. The images of the basic morphisms are just the structure operations of the Hopf algebra $H$.

For any two Hopf algebras $H_1$ and $H_2$ one can define the set $\HomH(H_1;H_2)$ of natural transformations between the functors $\mathfrak{F}_{H_1}$ and $\mathfrak{F}_{H_2}$, which are viewed as functors and not as monoidal functors. (If one considers these as monoidal functors, this will lead to the set of Hopf algebra morphisms.)

For any group $G$ the group algebra $\F [G]$ has a canonical Hopf algebra structure and the previous definition of $\HomH(G,H)$ is the same as $\HomH(\F [G];H)$. Note that since $\F [G]$ is a cocommutative Hopf algebra, the set $\HomH(\F [G];H)$ will be very small unless $H$ is also cocommutative.
\end{remark}
\begin{proposition}
$\HomH(G;H)$ is an $\Aut(G)$-module via the action
$$
\psi\cdot\varphi_k(g_1,\ldots,g_k)=\varphi_k(g_1^{\psi},\ldots,g_k^{\psi})
$$
for any $\psi\in \Aut(G)$, and where the notation $g^\psi=\psi^{-1}(g)$ denotes the right action of $\Aut(G)$ on $G$.
\end{proposition}
\begin{proof}
It is easy to see that  $\{\psi\cdot\varphi_k\}$ satisfies the above conditions. Also
$$
(\psi_1\psi_2)\cdot \varphi_k(g_1,\ldots,g_k)=\psi_1\cdot(\psi_2\cdot \varphi_k),
$$
so it does define a group action.
\end{proof}

\begin{proposition}
\label{prop:into}
Let $g_1,\ldots, g_n$ be a generating set for $G$. Then $\varphi\in \HomH(G;H)$ is completely determined by
$\varphi_n(g_1,\ldots,g_n)\in H^{\otimes n}.$
\end{proposition}
\begin{proof}
Let $w_1,\cdots,w_k$ be words in the generators and their inverses representing elements of $G$: $w_i=w_{i1}\cdots w_{i|w_i|}$ for $w_{ij}\in\{g_1^{\pm 1}\ldots, g_n^{\pm 1}\}$.

Suppose $g_i$ appears $k_i$ times overall and $g_i^{-1}$ appears $\ell_i$ times overall. Let $N$ be the sum of the lengths of the words $w_i$. Let $\sigma\in \sym{N}$ be the permutation rearranging the word $g_1^{k_1}g_1^{-\ell_1}\cdots g_n^{k_n}g_n^{-\ell_n}$ into the word $w_1w_2\cdots w_n$. By the above axioms, we have
$$
\left(\bigotimes_{i=1}^n \Delta^{k_i+\ell_i-1}\right)(\varphi_n(g_1,\ldots, g_n))= \varphi_{N}(g_1,\ldots,g_1,\ldots,g_n,\ldots,g_n).
$$
Here we adopt the conventions that $\Delta^{-1}=\epsilon$ and $\Delta^0=\id$ the first of which takes care of the case when one of the generators does not appear in any word.
Applying $\bigotimes_{i=1}^k(\id^{k_i}\otimes S^{\otimes \ell_i})$, this gives us
$$
\varphi_{N}(g_1,\ldots,g_1,g_1^{-1},\ldots,g_1^{-1},\ldots,g_n,\ldots, g_n,g_n^{-1},\ldots,g_n^{-1}).
$$
Applying the permutation $\sigma$ we get
$$
\varphi_{N} (w_{11},\ldots,w_{1|w_1|},\ldots,w_{n1},\ldots,w_{n|w_n|})
$$
and applying $\otimes_{i=1}^k m^{|w_i|-1} $, we get $\varphi_k(w_1,\ldots,w_k)$, here taking the convention that $m^{-1}=\epsilon$ and $m^0=\id$ the first of which takes care of the case that one of the words is empty.
Thus $\varphi_k(w_1,\ldots,w_k)$ is determined by $\varphi_n(g_1,\ldots, g_n)$. Indeed, we have just proved the formula
$$
\varphi_k(w_1,\ldots,w_k)=\left(\bigotimes_{i=1}^km^{|w_i|-1}\right)\circ \sigma\circ\left(\bigotimes_{i=1}^k(\id^{k_i}\otimes S^{\otimes \ell_i})\right)\circ\left(\bigotimes_{i=1}^n \Delta^{k_i+\ell_i-1}\right)\varphi_n(g_1,\ldots,g_n)
$$
\end{proof}

For any group $G$ with chosen generating set $g_1,\ldots, g_n$, there is an evaluation map
$$
\begin{array}{rcl}
\HomH(G;H) & \to & H^{\otimes n} \\
\varphi &\mapsto & \varphi_n(g_1,\ldots,g_n).
\end{array}
$$
The previous proposition shows that the evaluation map is always injective.

\begin{theorem}
\label{thm:onto}
If $G=F_n$ is the free group, then the evaluation map $\HomH(F_n;H)\to H^{\otimes n}$ is surjective, hence an isomorphism.
\end{theorem}
\begin{proof}
Let $x_1,\ldots,x_n$ be the standard generating set for $F_n$. Let $\mathbf{h}\in H^{\otimes n}$, and let $(w_1,\ldots,w_k)\in G^k$ be a $k$-tuple of non-empty words. Using the same notation as in the previous proposition define
$$
\varphi_k(w_1,\ldots,w_k):=\left(\bigotimes_{i=1}^km^{|w_i|-1}\right)\circ \sigma\circ\left(\bigotimes_{i=1}^k(\id^{k_i}\otimes S^{\otimes \ell_i})\right)\circ\left(\bigotimes_{i=1}^n \Delta^{k_i+\ell_i-1}\right)\mathbf h.
$$
Note that  this definition implies that $\varphi_n(x_1,\ldots,x_n)=\mathbf{h}$, so $\varphi\mapsto\mathbf{h}$ under the evaluation map if it is a well-defined element of $\HomH(G;H)$.

If $\mathbf{h}=h^1\otimes\cdots\otimes h^n$ is a pure tensor, $\varphi_k(w_1,\ldots,w_k)$ can be calculated as follows. For the $k$th occurrence of the generator $x_i^{\pm 1}$ in all words, replace it by $h^i_{(k)}$ or $S(h^i_{(k)})$ depending on the sign, and then take the tensor product of this replacement for each word. So, for example:
$$\varphi_2(x_2^2x_3,x_2x_3^{-1}x_1)=h^2_{(1)}h^2_{(2)}h^3_{(1)}\otimes h^2_{(3)}S(h^3_{(2)})h^1_{(1)}.$$

The fact that this is well-defined amounts to showing that inserting a pair $x_jx_j^{-1}$ or $x_j^{-1}x_j$ in one of the $w_i$s does not change the above element. This follows from the statement
\begin{align*}
h_{(1)}S(h_{(2)})\otimes h_{(3)}\otimes\cdots\otimes h_{(k)}&=S(h_{(1)})h_{(2)}\otimes h_{(3)}\otimes\cdots\otimes h_{(k)}\\
&=\eta(1)\otimes h_{(1)}\otimes\cdots\otimes h_{(k-2)},
\end{align*}
and this is just the statement
\begin{align*}
(m^2\otimes \id^{k-2})(\id\otimes S\otimes \id^{k-2})\Delta^k(h)&=(m^2\otimes\id^{k-2})(S\otimes\id\otimes\id^{k-2})\Delta^k(h)\\
&=\eta \otimes \Delta^{k-2}(h)
\end{align*} which follows from the antipode axiom.

Next, we must show that the above definition satisfies all of the appropriate conditions in Definition \ref{def:homh}, which is not hard. We give a 
sketch of the
proof of condition (5) and leave the rest to the reader.
We need to show that
$$
\varphi_k(w,w,w_1,\ldots,w_{k-2})=(\Delta\otimes \id^{k-2})\varphi_{k-1}(w,w_1,\ldots,w_{k-2}).
$$
Suppose that $w=x_{i_1}^{\epsilon_1}\ldots x_{i_m}^{\epsilon_m}$, where $\epsilon_i=\pm 1$.
Let $r_i=\frac{1}{2}(1-\epsilon_i)$.
Then
$$\varphi_{k-1}(w,w_1,\ldots,w_{k-2})=S^{r_1}(h^{i_1}_{(j_1)})  \cdots S^{r_m}(h^{i_k}_{(j_m)})\otimes \cdots,$$
for some indices $j_i$.
So $$(\Delta\otimes \id^{k-2})\varphi_{k-1}(w,w_1,\ldots,w_{k-2})=\Delta(S^{r_1}(h^{i_1}_{(j_1)})  \cdots S^{r_m}(h^{i_k}_{(j_m)}))\otimes \cdots,$$
and the proof is finished by noticing that
$$\Delta\left(S^{r_1}(h^{i_1}_{(j_1)})  \cdots S^{r_m}(h^{i_k}_{(j_m)})\right)= S^{r_1}(h^{i_1}_{(j_1)})  \cdots S^{r_m}(h^{i_k}_{(j_m)})\otimes S^{r_1}(h^{i_1}_{(j_1')})  \cdots S^{r_m}(h^{i_k}_{(j_m')}),$$
where $j_1',\ldots,j_m'$ are a new set of indices disjoint from $j_1,\ldots,j_m$. This follows since $\Delta$ is a homomorphism: $\Delta(ab)=a_{(1)}b_{(1)}\otimes a_{(2)}b_{(2)}$, and that it commutes with $S$: $\Delta(S(a))=S(a_{(1)})\otimes S(a_{(2)})$.

\end{proof}

\begin{definition}
Theorem~\ref{thm:onto} implies that $\Aut(F_n)$ acts on $H^{\otimes n}$ for any cocommutative Hopf algebra $H$. Let $\rho_{H}\colon \Aut(F_n)\to\End_{\F}(H^{\otimes n})$ denote this action.
\end{definition}

\begin{remark}
Unwinding the definitions one can see how elements of $\Aut(F_n)$ act on $H^{\otimes n}$. The symmetric group $\sym{n}$ which permutes the generators acts by permuting the factors in the tensor product; inverting a generator acts as the antipode on the corresponding component of the tensor product.
Finally the element $\eta\colon g_1 \mapsto g_2^{-1}g_1, g_2 \mapsto g_2^{-1}$ acts as
$(m \otimes \id^{n-1} )(\tau\otimes\id^{n-2})(\id \otimes \Delta \otimes \id^{n-2})(\id \otimes S\otimes \id^{n-2})$.
Figure~\ref{action} depicts $\eta$ when $n=2$.
Since these elements generate $\Aut(F_n)$ they determine the action of the whole group. One can directly check that these transformations satisfy the relations between these generators of $\Aut(F_n)$, using the explicit presentation from~\cite{AFV}.
Figure~\ref{action} shows an example of one such relation, and we challenge the reader to show the given relation (and the others) using graphical calculus.
\end{remark}

\begin{remark}
\label{rem:action-nt}
The proofs of Theorem~\ref{thm:onto} and Proposition~\ref{prop:into} can be rephrased to give  that for any functor $\mathcal{F}: \mathcal{HOPF} \to \mathcal{C}$ there is a bijection between the set of natural transformations between the functors $\mathfrak{F}_{\F[F_n]}$ and $\mathcal{F}$ with the object $\mathcal{F}(n)$, provided that the functor $\mathcal{F}$ is ``cocommutative''. Therefore there is a canonical action of $\Aut(F_n)$ on $\mathcal{F}(n)$,
which is given by the formula described above.
\end{remark}

\bigskip

Consider the case of the universal enveloping algebra $H=U(\mathfrak g)$. According to the PBW isomorphism, this is isomorphic (as a coalgebra) to $\Sym(\mathfrak g)$. Thus
$$
H^{\otimes n}\cong \Sym(\mathfrak g)^{\otimes n}\cong \Sym(\mathfrak g \otimes \F^n).
$$
Now clearly $\GL_n(\Z)$ acts on the latter space, and since $\Aut(F_n)$ and $\Out(F_n)$ surject onto $\GL_n(\Z)$, we also get an action of these groups on $H^{\otimes n}$, denoted $\rho_{S}\colon \Aut(F_n)\to \End_{\F}(H^{\otimes n})$. It is worthwhile to note that the representation $\rho_H$ constructed in this section is not the same as the representation $\rho_S$ defined via the PBW isomorphism. Indeed, $\rho_H$ does not in general factor through $\GL_n(\Z)$. Even in the case of $n=2$, when $\Out(F_2)\cong \GL_2(\Z)$, we will see in Remark~\ref{rem:noextension}
that $\rho_H$ does not in general extend to a representation of $\GL_2(\F)$, despite the fact that $\rho_S$ clearly does extend.

However, there is a similarity between these actions. The filtration $\{U_i\}$ on the algebra $H=U(\mathfrak{g})$ induces
a filtration $\{V_i\}$ of $H^{\otimes n}$ where
$$
V_i = \sum_{p_1+ \cdots + p_n = i} U_{p_1} \otimes \cdots \otimes U_{p_n} \subset H^{\otimes n}
$$
The PBW theorem implies that $V_{i}/V_{i-1} \cong \Sym^{i}(\mathfrak g \otimes \F^n)$. It is easy to verify that
the spaces $V_i$ are preserved under that action of $\Aut(F_n)$ and the induced action of  $\Aut(F_n)$
on the $V_{i}/V_{i-1}$ factors coincides with the natural action of $\GL_n(\Z)$ on $\Sym^{i}(\mathfrak g \otimes \F^n)$.

Thus, if one views $\Sym(\mathfrak g \otimes \F^n)$ as the associated graded module of $H^{\otimes n}$ the natural action of $\Aut(F_n)$ via $\GL_n(\Z)$ is the associated graded of the action $\rho_H$.

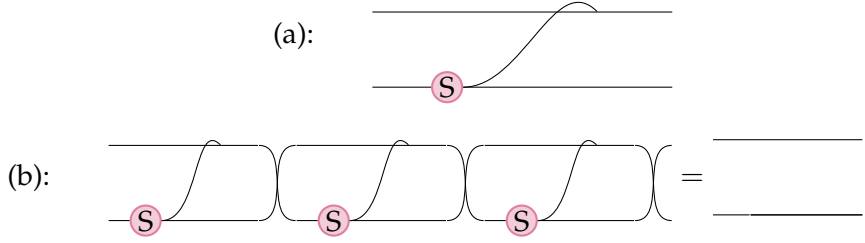
\begin{figure}
$$
\text{(a):\hspace{2em}}
\begin{minipage}{3cm}
\begin{tikzpicture}
\node[empty](A){};
\node[empty](B)[right of=A]{};
\node[empty](C)[right of =B]{};
\node[empty](D)[right of= C]{};
\node[empty](E)[right of =D]{};
\node[empty](F)[below of =A]{};
\node[antipode](G)[right of =F]{S} edge (F);
\node[empty](H)[right of =G]{};
\node[empty](I)[right of =H]{};
\node[empty](J)[right of =I]{} edge (G);
\draw (A) to (E);
\draw (G) to[out=0, in =135] (D);
\end{tikzpicture}
\end{minipage}
$$
$$\text{(b):\hspace{2em}}
\begin{minipage}{2.5cm}
\begin{tikzpicture}
\node[empty](A){};
\node[empty](B)[right of=A, node distance=.5cm]{};
\node[empty](C)[right of =B, node distance=.5cm]{};
\node[empty](D)[right of= C, node distance=.5cm]{};
\node[empty](E)[right of =D, node distance=.5cm]{};
\node[empty](F)[below of =A]{};
\node[antipode](G)[right of =F, node distance=.5cm]{S} edge (F);
\node[empty](H)[right of =G, node distance=.5cm]{};
\node[empty](I)[right of =H, node distance=.5cm]{};
\node[empty](J)[right of =I, node distance=.5cm]{} edge (G);
\draw (A) to (E);
\draw (G) to[out=0, in =135] (D);
\node[empty] (K) [right of =E, node distance=.5cm]{};
\node[empty] (L) [right of =J, node distance=.5cm]{};
\draw (E) to[out=0, in =180] (L);
\draw (J) to[out=0, in=180] (K);
\end{tikzpicture}
\end{minipage}
\begin{minipage}{2.5cm}
\begin{tikzpicture}
\node[empty](A){};
\node[empty](B)[right of=A, node distance=.5cm]{};
\node[empty](C)[right of =B, node distance=.5cm]{};
\node[empty](D)[right of= C, node distance=.5cm]{};
\node[empty](E)[right of =D, node distance=.5cm]{};
\node[empty](F)[below of =A]{};
\node[antipode](G)[right of =F, node distance=.5cm]{S} edge (F);
\node[empty](H)[right of =G, node distance=.5cm]{};
\node[empty](I)[right of =H, node distance=.5cm]{};
\node[empty](J)[right of =I, node distance=.5cm]{} edge (G);
\draw (A) to (E);
\draw (G) to[out=0, in =135] (D);
\node[empty] (K) [right of =E, node distance=.5cm]{};
\node[empty] (L) [right of =J, node distance=.5cm]{};
\draw (E) to[out=0, in =180] (L);
\draw (J) to[out=0, in=180] (K);
\end{tikzpicture}
\end{minipage}
\begin{minipage}{2.5cm}
\begin{tikzpicture}
\node[empty](A){};
\node[empty](B)[right of=A, node distance=.5cm]{};
\node[empty](C)[right of =B, node distance=.5cm]{};
\node[empty](D)[right of= C, node distance=.5cm]{};
\node[empty](E)[right of =D, node distance=.5cm]{};
\node[empty](F)[below of =A]{};
\node[antipode](G)[right of =F, node distance=.5cm]{S} edge (F);
\node[empty](H)[right of =G, node distance=.5cm]{};
\node[empty](I)[right of =H, node distance=.5cm]{};
\node[empty](J)[right of =I, node distance=.5cm]{} edge (G);
\draw (A) to (E);
\draw (G) to[out=0, in =135] (D);
\node[empty] (K) [right of =E, node distance=.5cm]{};
\node[empty] (L) [right of =J, node distance=.5cm]{};
\draw (E) to[out=0, in =180] (L);
\draw (J) to[out=0, in=180] (K);
\end{tikzpicture}
\end{minipage}
=
\begin{minipage}{2.0cm}
\begin{tikzpicture}
\node[empty](A){};
\node[empty](B)[right of=A, node distance=.5cm]{};
\node[empty](C)[right of =B, node distance=.5cm]{};
\node[empty](D)[right of= C, node distance=.5cm]{};
\node[empty](E)[right of =D, node distance=.5cm]{};
\node[empty](F)[below of =A]{};
\node[empty](G)[right of =F, node distance=.5cm]{} edge (F);
\node[empty](H)[right of =G, node distance=.5cm]{};
\node[empty](I)[right of =H, node distance=.5cm]{};
\node[empty](J)[right of =I, node distance=.5cm]{} edge (G);
\draw (A) to (E);
\draw (G) to (J);
\end{tikzpicture}
\end{minipage}
$$
\caption{(a) The action of $\eta\in \Aut(F_2)$ on $H^{\otimes 2}$. (b) Puzzle: Show $(\sigma_{12}\eta)^3=\id$ using graphical calculus.}\label{action}
\end{figure}

\begin{remark}
For any finitely generated group $\Gamma$ the space $\HomH(\Gamma;U(\mathfrak{g}))$ is dual to the space of jets of functions on the representation variety $\Hom(\Gamma;G)$ near the trivial representations, where $G$ is a Lie group with Lie algebra $\mathfrak{g}$, for details see~\cite{Kas}.\end{remark}

\subsection{$\overline{H^{\otimes n}}$ and an action of $\Out(F_n)$}

\begin{definition}
The Hopf algebra $H$ acts on $H^{\otimes n}$ via \emph{conjugation}. That is, suppose $h\in H$ and $\Delta^{2n-1}(h)=h_{(1)}\otimes h_{(2)}\otimes\cdots\otimes h_{(2n-1)}\otimes h_{(2n)}$, using Sweedler notation. Then define
$$
h\circledast (h_1\otimes\cdots\otimes h_n)= h_{(1)}h_1S(h_{(2)})\otimes\cdots\otimes h_{(2n-1)}h_n S(h_{(2n)}).
$$
The Hopf algebra axioms imply that this operations is associative, i.e.,
$$
(hh')\circledast (h_1\otimes\cdots\otimes h_n) =
h\circledast \left( h' \circledast (h_1\otimes\cdots\otimes h_n) \right).
$$

Let $C_{\circledast}$ denote the map $H^{\otimes (n+1)} \to H^{\otimes n}$ given by
$h \otimes h_1\otimes\cdots\otimes h_n \to h \circledast (h_1\otimes\cdots\otimes h_n)$ and
let $h_{\circledast}$ denote the map $H^{\otimes n} \to H^{\otimes n}$ given by
$h_1\otimes\cdots\otimes h_n \to h\circledast (h_1\otimes\cdots\otimes h_n)$.

Let $\overline{H^{\otimes n}}$ be the quotient of $H^{\otimes n}$ by the subspace $\widetilde{H^{\otimes n}}$ spanned by elements of the form
$$
(h-\eta\epsilon(h))\circledast(h_1\otimes \cdots\otimes h_n),
$$
i.e., this is the maximal quotient of $H^{\otimes n}$ where the conjugation action of $H$ factors through the counit.
\end{definition}
\begin{lemma}
\label{lem:conjonHomH}
Let $\varphi\in \HomH(G; H)$ and define
$h \circledast \varphi$ as the sequence $\{(h \circledast \varphi)_n\}$, where
$$
(h \circledast \varphi)_n(g_1,\dots, g_n)  = h \circledast \varphi_n(g_1,\dots, g_n),
$$
i.e., $h \circledast \varphi = h_{\circledast} \circ \varphi$. Then
$h \circledast \varphi \in \HomH(G,H)$.
\end{lemma}
\begin{proof}
One just verifies the axioms. E.g.,  axiom (5):
\begin{align*}
h_\circledast\circ\varphi_k(g,g,g_1,\ldots,g_{k-2})&= h_{\circledast}\circ (\Delta\otimes \id^{k-2})\varphi_{k-1}(g,g_1,\ldots,g_{k-2})\\
&\overset{?}{=} (\Delta\otimes \id^{k-2})h_{\circledast} \circ \varphi_{k-1}(g,g_1,\ldots,g_{k-2}).
\end{align*}
So it suffices to show that $h_\circledast\circ \Delta=\Delta\circ h_\circledast$.  Let $\cdot$ denote componentwise multiplication in $H\otimes H$: $(h_1\otimes h_2)\cdot(h_3\otimes h_4)=h_1h_3\otimes h_2h_4$. Then the compatibility of $\Delta$ and $m$ implies that $\Delta(h_1h_2)=\Delta(h_1)\cdot\Delta (h_2)$. Now,
\begin{align*}
h_\circledast\circ \Delta(k)&= h_{(1)}k_{(1)}S(h_{(2)})\otimes h_{(3)}k_{(2)}S(h_{(4)})\\
&=(h_{(1)}\otimes h_{(3)})\cdot( k_{(1)}\otimes k_{(2)})\cdot (S(h_{(2)})\otimes S(h_{(4)}))\\
&= \Delta(h_{(1)}) \cdot \Delta (k) \cdot \Delta(S(h_{(2)}))\\
&=\Delta (h_{(1)})kS(h_{(2)})\\
&=\Delta\circ h_\circledast(k)
\end{align*}
completing the proof.
\end{proof}

\begin{remark}
\label{rem:conjugation}
Let $\varphi\in \HomH(G; H)$ be the element which sends $(g_0,g_1\dots, g_n)$ to $h \otimes h_1 \otimes\cdots\otimes h_n$, then $\varphi$ maps $(g_1^{g_0}, \dots g_n^{g_0})$ to $h\circledast (h_1\otimes\cdots\otimes h_n)$ which justifies the term \emph{conjugation}.
\end{remark}

\begin{lemma}\label{lem:primgen}
Suppose that $H=U(\mathfrak g)$. Then
$\widetilde{H^{\otimes n}}$ is the subspace of $H^{\otimes n}$ spanned by $X\circledast(h_1\otimes \cdots\otimes h_n)$ for $X\in\mathfrak g$.
\end{lemma}
\begin{proof}
By definition $\widetilde{H^{\otimes n}}$ is the $\F$-span of elements of the form $(h-\eta\epsilon(h))\circledast(h_1\otimes\cdots\otimes h_n)$. Let $h=h_0+h_1+\cdots+h_k$ where $h_i$ is a product of $i$ elements of $\mathfrak g$. Then the above element can be written as
$$(h_1+\cdots+h_k)\circledast (h_1\otimes\cdots\otimes h_n)=\sum_{i\geq 1}h_i\circledast(h_1\otimes\cdots\otimes h_n).$$
Finally, the terms $h_j\circledast(h_1\otimes\cdots\otimes h_n)$ can be gotten by repeated conjugation by elements of $\mathfrak g$, since
$$(hk)\circledast(h_1\otimes\cdots\otimes h_n)=h\circledast(k\circledast(h_1\otimes\cdots\otimes h_n)).$$
\end{proof}

\begin{lemma}
If $H=U(\mathfrak{g})$ 
then $\overline{H^{\otimes 1}}\cong H/[H,H]$, where $[H,H]$ is the $\F$-vector space spanned by commutators $h_1h_2-h_2h_1$. In particular, if $H=T(V)$, then
$$
\overline{H^{\otimes 1}}=\bigoplus_{k\geq 0} \left(V^{\otimes k}\right)_{\Z_k}
$$ is the space of cyclic words.
\end{lemma}
\begin{proof}
Suppose $v\in \mathfrak{g}$, i.e., $v$ is primitive, thus
$$
(v-\eta\epsilon(v))\circledast(h)= v \circledast h = vh-hv \in [H,H].
$$
The associativity of the operation implies that
$(h'-\eta\epsilon(h'))\circledast(h) \in [H,H]$ for any element $h' \in U(\mathfrak{g})$.
The other inclusion follows by a induction since $[H,H]$ is generated as a vector space
by $[h',h]$ where $h'$ is primitive.
%
\end{proof}

%

\begin{lemma}
The representation $\rho_H\colon\Aut(F_n)\to \End_{\F}(H^{\otimes n})$ induces a representation
$$
\rho_H\colon\Aut(F_n)\to \End_{\F}(\overline{H^{\otimes n}}).
$$
\end{lemma}

\begin{proof}
Observe that the actions of $\Aut(G)$ on $\HomH(G,H)$ commutes with the conjugation action of $H$, i.e.,
$$
\psi\cdot(h\circledast (h_1\otimes\cdots\otimes h_n))=h\circledast (\psi\cdot (h_1\otimes\cdots\otimes h_n)).
$$
Therefore
$$
\psi\cdot((h - \eta\epsilon(h) ) \circledast (h_1\otimes\cdots\otimes h_n))= (h - \eta\epsilon(h) )\circledast (\psi\cdot (h_1\otimes\cdots\otimes h_n)),
$$
i.e.,
the kernel $\widetilde{H^{\otimes n}}$ of $H^{\otimes n} \to \overline{H^{\otimes n}}$ is invariant under the action of $\Aut(F_n)$.
\end{proof}

\begin{remark}
It is easy to see the map $m \otimes \id^n : H^{\otimes (n+2)} \to H^{\otimes (n+1)}$ induces a map
$\overline{m \otimes \id^n} : \overline{H^{\otimes (n+2)}} \to \overline{H^{\otimes (n+1)}}$ and similarly for the multiplications. I.e.,  there is a canonical functor $\mathfrak{F}_{\overline H} : \mathcal{HOPF} \to \overline{\End_{\F}}(H)$ which sends the object $n$ to $\overline{H^{\otimes n}}$.
Using Remark~\ref{rem:action-nt} this leads to an action of $\Aut(F_n)$ on $\overline{H^{\otimes n}}$.
%
\end{remark}


\begin{proposition}
$\mathrm{Inn}(F_n)$ acts trivially on $\overline{H^{\otimes n}}$, which is therefore an $\Out(F_n)$-module with induced representation $\rho_H\colon \Out(F_n)\to \overline{H^{\otimes n}}$.
\end{proposition}
\begin{proof}
As before it is easier instead of $H^{\otimes n}$ to consider $\HomH(F_n;H)$. Let $\varphi \in \HomH(F_n;H)$ and
$\psi \in \mathrm{Inn}(F_n)$ is given by conjugation with $g$. As mentioned in Remark~\ref{rem:conjugation}
$$
(\psi \cdot \varphi)_n(g_1,\dots,g_n) =
\varphi_n(g_1^g,\dots,g_n^g)=
C_\circledast(\varphi_{n+1}(g, g_1,\dots,g_n)).
$$
By the definition of $\HomH$ we have
$\epsilon \otimes \id^{\otimes n}(\varphi_{n+1}(g, g_1,\dots,g_n)) = \varphi_n(g_1,\dots,g_n)$
Therefore
$
(\psi \cdot \varphi)_n(g_1,\dots,g_n) - \varphi_n(g_1,\dots,g_n)  =
(C_\circledast - \epsilon \otimes \id^{\otimes n})\varphi_{n+1}(g, g_1,\dots,g_n)
$
is in $\widetilde{H^{\otimes n}}$ this the action of $\mathrm{Inn}(F_n)$ on $\overline{H^{\otimes n}}$ is trivial.
\end{proof}

\begin{remark}
Let $\widehat{H^{\otimes n}}$ be the maximal quotient of $H^{\otimes n}$ with an induced
action of $\Out(F_n)$. This is actually larger than $\overline{H^{\otimes n}}$ for general $H$.
For example, when $n=1$, the group $\mathrm{Inn}(F_1)$ is trivial and  $\Out(F_1)= \Z_2$ acts on $H^{\otimes 1}=H$ via the antipode. So $\widehat{H^{\otimes 1}}=H^{\otimes 1}$.
Yet $\overline{H^{\otimes 1}}=H/[H,H]$ is a nontrivial quotient (for most $H$).

In general, consider the exact sequence $1\to F_n\to \Aut(F_n)\to \Out(F_n)\to 1$. Then  $\widehat{H^{\otimes n}}$ will be the coinvariants of $H^{\otimes n}$ under the action of $\rho_H|_{F_n}$. For example, when $n=2$,
$$
\widehat{H^{\otimes 2}}=H^{\otimes 2}/\{h_1\otimes h_2=h_1'\otimes S(h_1'')h_2 h_1''', h_1\otimes h_2=S(h_2')h_1h_2''\otimes h_2'''\}.
$$
A direct computation shows that there is a difference between $\widehat{T(V)^{\otimes 2}}$ and $\overline{T(V)^{\otimes 2}}$.

We expect that the same happens for $n\geq 2$.

%
\end{remark}

Recall that if $H = U(\mathfrak{g})$ the PBW isomorphism gives that the associated graded of the module $H^{\otimes n}$ is
$\Sym(\mathfrak{g} \otimes \F^n)$. This leads to the natural question what is the associated graded of $\overline{H^{\otimes n}}$, which we now discuss.

Consider the adjoint action of $\mathfrak{g}$ on itself, i.e., $\ad(X)(Y) =[X,Y]$.
This induces an adjoint action of $\mathfrak{g}$ on $\mathfrak{g} \otimes \F^n$ which can be extended to
action on $\Sym(\mathfrak{g} \otimes \F^n)$ using the Leibniz rule.

\begin{lemma}
The associated graded of the module $\overline{H^{\otimes n}}$ is the quotient of $\Sym(\mathfrak{g} \otimes \F^n)$
by the image of the adjoint action.
Moreover for any $f \in \Sym(\mathfrak{g} \otimes \F^n)$, we have $\sigma(\ad(X)(f)) \in \widetilde{H^{\otimes n}}$.
Here we abuse the notation and use $\sigma$ to denote the map $\Sym(\mathfrak{g} \otimes \F^n) \to H^{\otimes n}$.
\end{lemma}
\begin{proof}
The filtration $V_i$ of $H^{\otimes n}$ induces a filtration $\overline{V_i}$ of $\overline{H^{\otimes n}}$. Indeed
let $\widetilde{V_i}=V_i\cap \widetilde{H^{\otimes n}}$. Then $\overline{V_i}=V_i/\widetilde{V_i}$.

We wish to show that the isomorphism $$\sigma_i\colon \Sym^i(\mathfrak g\otimes \F^n)\overset{\cong}{\longrightarrow} V_i/V_{i-1}$$ takes the subspace $\im(\ad)$ to $\widetilde{V_i}/\widetilde{V}_{i-1}$, which would then imply that it induces an isomorphism
$$ \overline{\sigma_i}\colon \Sym^i(\mathfrak g\otimes \F^n)/\im(\ad)\overset{\cong}{\longrightarrow} \overline{V_i}/\overline{V_{i-1}}$$

If $X\in\mathfrak g$, then observe that $\sigma(\ad(X)(f)) = X \circledast \sigma(f)$ which together with $\epsilon(X)=0$ implies that
$\sigma(\ad(X)(f)) \in \widetilde{H^{\otimes n}}$, and thus that $\sigma_i(\im(\ad))\subset \widetilde{V_i}$.
On the other hand, by Lemma~\ref{lem:primgen},  $\widetilde{V_i}$ is generated by elements $X\circledast (h_1\otimes \cdots\otimes h_n)$ for $X\in\mathfrak g$. Thus $\sigma_i(\im(\ad))=\widetilde{V_i}$, since we can choose $f$ such that $\sigma_i(f)=h_1\otimes \cdots\otimes h_n$ and map $\sigma(\ad(X)(f))$ to
$X\circledast (h_1\otimes \cdots\otimes h_n)$.
\end{proof}

Let $\overline{\Sym(\mathfrak{g} \otimes \F^n)}$ the quotient of $\Sym(\mathfrak{g} \otimes \F^n)$ by the image of the adjoint action and let  $\overline{\sigma}$ denote the map $\overline{\Sym(\mathfrak{g} \otimes \F^n)}$ to
$\overline{H^{\otimes n}}$.
Clearly the action of $\GL_n(\Z)$ on $\Sym(\mathfrak{g} \otimes \F^n)$ give rise to an action
$\overline{\Sym(\mathfrak{g} \otimes \F^n)}$. It is clear that this is associated graded of the action of
$\Out(F_n)$ on $\overline{H^{\otimes n}}$.

In the case $n=2$ the group $\Out(F_2)$ coincides with $\GL_2(\Z)$ and one can expect that
$\overline{\Sym(\mathfrak{g} \otimes \F^2)}$ and $\overline{H^{\otimes 2}}$ are isomorphic as $\GL_2(\Z)$ modules.
This is not the case for general $\mathfrak g$ however:

\begin{theorem}
\label{thm:class2}
If $\mathfrak{g}$ is nilpotent of class $2$ then $\overline{\sigma}$  is also a  $\GL_2(\Z)$ module isomorphism.
\end{theorem}
\begin{proof}
It is sufficient to verify that $\overline{\sigma}$ commutes with the generators of the group $\Out(F_2)$. This is trivial
for the elements of the signed permutation group, so we only need to verify for the element $\eta$ which fixes $g_2$ and sends $g_1$ to $g_2^{-1}g_1$.
By Claim~\ref{cl:powers} the image of $\Sym^i(\mathfrak{g}) \otimes \Sym^j(\mathfrak{g})$ under $\sigma$ inside
$H^{\otimes 2}$ is generated by elements of the form
$X^i \otimes Y^j$, with $X,Y \in \mathfrak{g}$.
By the definition of $\eta$ and the action of $\Aut(F_2)$ we have
$$
\eta(X^i \otimes Y^j) = (m \otimes \id)(S \otimes \id\otimes \id)(\tau \otimes \id)(\id \otimes \Delta)(X^i \otimes Y^j),
$$
since $\Delta(Y^j) = \sum_k {j \choose k} Y^k \otimes Y^{j-k}$ this gives
$$
\eta(X^i \otimes Y^j) = \sum_k (-1)^k {j \choose k} Y^k X^i \otimes Y^{j-k}.
$$

By Lemma~\ref{lm:class2} there exists constants $d_{i,p,k}$ such that
$$
Y^k X^i = \sigma(X^iY^k) + \sum  d_{i,p,k} \sigma(X^{i-p}Y^{k-p}[X,Y]^p).
$$
Observe that $\sigma(X^{i-p}Y^{k-p}[X,Y]^p) \otimes Y^{j-k}$ lies in $\widetilde{H^{\otimes 2}}$ because
$$
(i-p+1)\sigma(X^{i-p}Y^{k-p}[X,Y]^p) \otimes Y^{j-k} = Y \circledast (\sigma(X^{i-p+1}Y^{k-p}[X,Y]^{p-1}) \otimes Y^{j-k}).
$$
Therefore
$$
\eta(X^i \otimes Y^j) =
\sum_k (-1)^k {j \choose k} \sigma(Y^k X^i) \otimes Y^{j-k}
\,\, \mathrm{mod} \,\, \widetilde{H^{\otimes 2}},
$$
where the right side is the same as the action of $\left(\begin{matrix}1 & -1 \\ 0 & 1\end{matrix}\right)$ on
$X^i \otimes Y^j$.
Thus the action of $\overline{\sigma}$ commutes with the action of $\eta$.
\end{proof}
\begin{remark}
\label{rem:noextension}
This is not true for an arbitrary Lie algebra $\mathfrak{g}$ -
since the elements
$\scriptstyle \left(\begin{smallmatrix}1 & 1 \\ 0 & 1\end{smallmatrix}\right)$ and
$\scriptstyle \left(\begin{smallmatrix}1 & 0 \\ 1 & 1\end{smallmatrix}\right)$
act as locally unipotent elements on $\overline{H^{\otimes 2}}$ one can define
$$
E = \log \rho_H \left(\begin{matrix}1 & 1 \\ 0 & 1\end{matrix}\right)
\quad \mbox{and} \quad
F = \log \rho_H \left(\begin{matrix}1 & 0 \\ 1 & 1\end{matrix}\right)
$$
using formal power series. These series converge since
$ \rho_H \left(\begin{matrix}1 & 1 \\ 0 & 1\end{matrix}\right) -\id $
is a locally nilpotent operator on $U(\mathfrak{g})\otimes U(\mathfrak{g})$.
If the action of $\SL_2(\Z)$ extends to $\SL_2(\F)$ then $E,F$ would satisfy
the defining relations of $\sl_2$. However it is not the case -- it is not hard to show that
$$
(\id \otimes \eta )\left([[E,F],E] - 2 E \right) (x^3 \otimes y^3) = 24 [[x,y],y][[x,y],x]
$$
and the last element is non-zero in $\overline{H^{\otimes 1}}$ for some $\mathfrak{g}$.
\end{remark}

\section{Graph homology}
\label{sec:graph_homology}

Recall from~\cite{CV} that one can define a graph complex $\cG_\cO$ for any cyclic operad $\cO$ by putting elements of $\cO\arity{|v|}$ at each vertex $v$ of a graph and identifying the i/o slots with the adjacent edges. In this definition, the graph may have bivalent vertices but no univalent or isolated vertices, since the operad $\cO$ is assumed not to have anything in arity $-1$ and $0$.
Similarly one can define the based graph complex $\cGB_\cO$, where the graphs have a distinguished vertex $b$ and the element at this vertex lies in $\cO\arity{|b|+1}$ and one of the i/o slots is associated with the ``base point''.

These complexes are graded by the number of vertices of the underlying graph and the boundary operator is induced by contracting edges of the underlying graph.  Let $\cG_{\cO}^{(n)}$ and $\cGB_{\cO}^{(n)}$ be the subcomplexes spanned by $\cO$-colored connected graphs of rank $n$.
In this section we will study $H_\bullet(\cG_{H\Lie})$ and $H_\bullet(\cGB_{H\Lie})$. It is clear that the rank $0$ parts of $\cG_{H\Lie}$ and $\cGB_{H\Lie}$ are trivial.

\begin{definition}
Let $\overline{\cG_{H\Lie}}$ denote the quotient of the graph complex for $H\Lie$ where the elements in $H$ are allowed to slide through the edges, i.e., the following graphs in $\cG_{H\Lie}$ are equivalent in $\overline{\cG_{H\Lie}}$.
\begin{center}
\begin{minipage}{5.5cm}
\begin{tikzpicture}
\node[hopf](a){$h$};
\node[empty](b)[left of=a, node distance=.8cm]{} edge (a);
\node[empty](c)[above left of=b, node distance=.8cm]{$\ddots$} edge (b);
\node[empty](d)[below left of=b, node distance=.8cm]{$\Ddots$} edge (b);
\node[empty](e)[right of=a, node distance=.8cm]{} edge (a);
\node[empty](f)[right of=e, node distance=1.2cm]{};
\begin{scope}[decoration={markings,mark = at position 0.5 with {\arrow{stealth}}}]
\draw[decorate] (e) to (f);
\end{scope}
\draw[densely dashed] (e) to (f);
\node[empty](g)[right of=f, node distance=.8cm]{} edge (f);
\node[empty](h)[above right of =g, node distance=.8cm]{$\Ddots$} edge (g);
\node[empty](i)[below right of =g, node distance=.8cm]{$\ddots$} edge (g);
\end{tikzpicture}
\end{minipage}
$=$\,\,\,
\begin{minipage}{5.5cm}
\begin{tikzpicture}
\node[hopf](a){$h$};
\node[empty](b)[right of=a, node distance=.8cm]{} edge (a);
\node[empty](c)[above right of=b, node distance=.8cm]{$\Ddots$} edge (b);
\node[empty](d)[below right of=b, node distance=.8cm]{$\ddots$} edge (b);
\node[empty](e)[left of=a, node distance=.8cm]{} edge (a);
\node[empty](f)[left of=e, node distance=1.2cm]{};
\begin{scope}[decoration={markings,mark = at position 0.5 with {\arrow{stealth}}}]
\draw[decorate] (f) to (e);
\end{scope}
\draw[densely dashed] (e) to (f);
\node[empty](g)[left of=f, node distance=.8cm]{} edge (f);
\node[empty](h)[above left of =g, node distance=.8cm]{$\ddots$} edge (g);
\node[empty](i)[below left of =g, node distance=.8cm]{$\Ddots$} edge (g);
\end{tikzpicture}
\end{minipage}
\end{center}
Similarly we can define $\overline{\cG_{H\Lie}^{(n)}}$, $\overline{\cGB_{H\Lie}}$  and $\overline{\cGB_{H\Lie}^{(n)}}$.
It is clear that the quotient map $\cG_{H\Lie} \to \overline{\cG_{H\Lie}}$ preserves the differential and induces a map between the homologies.
\end{definition}

\subsection{Comparison to hairy Lie graph homology}
In this section we compare our constructions with that of ``hairy graph homology"  ~\cite{CKV1, CKV2}.
Let $\hairy\Lie$  denote the hairy Lie graph complex, which is spanned by graphs whose vertices have elements of $\Lie_V$ coloring them. In other words,
\begin{proposition}\label{prop:hairyG}
There is an isomorphism of chain complexes $\hairy\Lie \cong \mathcal G_{\Lie_V}$.
\end{proposition}
\begin{proof}
This is close to the definition given in \cite{CKV2}.
\end{proof}

It is convenient also to calculate $ H_*(\mathcal G_{\Lie_V})$ in terms of the $H\Lie$ construction:
\begin{theorem}
For all $n\geq 1$, the chain complex $\mathcal G^{(n)}_{\Lie_V}$ is quasi-isomorphic to the chain complex $\mathcal G^{(n)}_{\Sym(V)\Lie}$.
\end{theorem}
\begin{proof}
See section 5.3 of \cite{CKV2}.
\end{proof}
One of the main results in~\cite{CKV1} gives that the first homology of the connected part of the graph complex of $\Lie_V$ is
\begin{theorem}[\cite{CKV1}]
$$
H_1(\mathcal G^{(r)}_{\Sym(V)\Lie} )=\begin{cases}
\bigoplus_{k \geq 0} \Sym^{2k+1}(V) & r=1\\
H^{\mathrm{vcd}}\left(\Out(F_r); \Sym(V \otimes \F^r) \right) & r>1
\end{cases}
$$
where the action of $\Out(F_r)$ on the symmetric algebra $ \Sym(V \otimes \F^r)$ is via the $\GL_r(\Z)$ action on $\F^r$.
The $r=1$ term corresponds to the Morita trace~\cite{Mo}.
\end{theorem}

In this paper, we consider what happens when we replace $\Sym(V)$ with $T(V)$ in the above construction.
 By Proposition~\ref{prop:Lie-iso}, this is closely related to $\overline{\Lie_V}$.
The rank $0$ part is trivial, since $H\Lie$ does not contain any elements of arity $-1$; the rank $1$ part is related to the Enomoto-Satoh trace~\cite{ES} as discussed in subsection~\ref{ssec:rank_1}.
One might guess that in higher rank it is sufficient to replace the term
$\Sym(V \otimes \F^r)$ with $T(V \otimes \F^r)$
but this is not quite correct since hairs associated with different edges need to commute. Imposing this relation, one obtains $T(V)^{\otimes r}$. However, it is far from clear how $\Out(F_r)$
acts on this space -- in fact it does not act, requiring us to take the quotient $\overline{T(V)^{\otimes r}}$ as we saw in section~\ref{sec:action_on_Hn}.

\subsection{Graph homology in rank 1}
\label{ssec:rank_1}

The graphs of rank $1$ can contain only bivalent vertices and consists of loops of vertices labeled by elements in $H$. Here are examples of typical elements in $\cG_{H\Lie,1}^{(1)}$ and $\cG_{H\Lie,2}^{(1)}$.
\begin{center}
\begin{tikzpicture}
\coordinate(a) at (0,0);
\coordinate(b) at (.6,0);
\coordinate(c) at (-.6,0);
\coordinate(d) at (0,1);
\coordinate(e) at (.6,1);
\coordinate(f) at (-.6,1);
\node[hopf, at=(a)] (A){$h$};
\draw (A) to (b);
\draw (A) to (c);
\draw (e) to (f);
\draw (c) to[in=180, out=180] (f);
\draw (b) to[in=0, out=0] (e);
\end{tikzpicture}
\hspace{3em}
\begin{tikzpicture}
\coordinate(a) at (0,0);
\coordinate(b) at (.6,0);
\coordinate(c) at (-.6,0);
\coordinate(d) at (0,1);
\coordinate(e) at (.6,1);
\coordinate(f) at (-.6,1);
\node[hopf, at=(a)] (A){$h_1$};
\node[hopf, at=(d)] (D){$h_2$};
\draw (A) to (b);
\draw (A) to (c);
\draw (e) to (D) to (f);
\draw (c) to[in=180, out=180] (f);
\draw (b) to[in=0, out=0] (e);
\end{tikzpicture}
\end{center}

%

Here is a typical element of $\cGB_{H\Lie,3}^{(1)}$:

\begin{center}
\begin{minipage}{5cm}
\begin{tikzpicture}
\node[operad](a){$o$};
\node[below of=a]{$\flat$} edge (a);
\node[hopf](b)[right of=a, node distance=1.5cm]{$h_1$} edge(a);
\node[hopf](c)[right of=b, node distance=1.5cm]{$h_2$} edge(b);
\node[empty](d)[above of=a]{};
\node[empty](e)[above of=c]{} edge (d);
\draw (c) to[out=0, in =0] (e);
\draw (a) to[out=180, in=180] (d);
\end{tikzpicture}
\end{minipage}
where\,\,\,
\begin{minipage}{7cm}
\begin{tikzpicture}
\node[operad](a){$o$};
\node[below of=a]{$\flat$} edge (a);
\node[empty](b)[right of=a, node distance=.7cm]{}edge(a);
\node[empty][left of=a, node distance=.7cm]{}edge(a);

\node[empty](c)[right of =b]{};
\node[hopf](d)[right of=c]{$h'$} edge (c);
\node[empty](e)[right of=d]{} edge (d);
\node[empty][below of=e]{$\flat$} edge (e);
\node[hopf](f)[right of=e]{$h''$} edge (e);
\node[empty][right of=f]{} edge(f);

\path (b) to node(eq){$=$} (c);
\end{tikzpicture}
\end{minipage}
\end{center}

\begin{theorem}
The homology $H_k(\cGB_{H\Lie,\cdot}^{(1)})=0$ if $k\geq 2$ and $H_1(\cGB_{H\Lie,\cdot}^{(1)}) = H/H^{\sym{2}}$, where $\sym{2}$ acts via the antipode.
\end{theorem}
\begin{proof}
As seen in the above picture, a graph in  $\cGB_{H\Lie,\cdot}^{(1)}$ consists of a sequence of elements of $H\Lie$ arranged in a circle. Away from the basepoint, these must have arity $1$, and so they must lie in $H\cong H\Lie\arity{2}$. The element of $H\Lie$ at the basepoint must lie in $H\Lie\arity{3}$, we can push the elements of $H$ away from the basepoint so that as we travel around the circle we see an element of $H$ followed by the basepoint hair followed by another element of $H$.  Ignoring the $\sym{2}$ graph automorphism, the graphical chain complex is the same as the shifted bar complex for the algebra $H$ with coefficients in itself~\cite[1.1.11]{Loday}. The bar complex is exact except in the bottom degree, so we have $H_k(\cGB_{H\Lie,\cdot}^{(1)})=0$ if $k\geq 2$.

To calculate the first homology, note that the $1$ chains are $(H\otimes H)_{\sym{2}}$, where $h_1\otimes h_2=-S(h_2)\otimes S(h_1)$. The $2$ chains are of the form $[H\otimes(H\otimes H)] _{\sym{2}}$, where the differential is defined by $d(h\otimes(h_1\otimes h_2))=h_1\otimes(h_2h)-(hh_1)\otimes h_2$. Thus $h_1\otimes h_2=h_2h_1\otimes 1$ modulo boundaries. Thus $H_1(\cGB_{H\Lie,\cdot}^{(1)})$ is a quotient of $H$. The relation $h\otimes 1=-S(1)\otimes S(h)$ implies that we are also quotienting by $h=-S(h)$.
 Thus we get the antisymmetric part $H/H^{\sym{2}}$.
\end{proof}

If there is no basepoint, then in addition to the $\sym{2}$-symmetry of the graph, there is a full dihedral group of symmetries. This leads to (skew) \emph{dihedral homology} of an algebra $H$. 

\begin{theorem}
$H_k(\cG_{H\Lie,\cdot}^{(1)}) = HD'_{k-1}(H)$.
In particular
$$H_1(\cG_{H\Lie,\cdot}^{(1)}) =HD'_0(H) \cong(\id - S)(H/[H,H]) .$$
\end{theorem}
\begin{proof}
This is similar to Proposition 6.2 in~\cite{CKV2}. The chain group $\cG_{H\Lie,k}^{(1)}$ is isomorphic to $ (H^{\otimes k})_{D_{2k}}$, and the boundary map is given by summing over pairwise multiplication of adjacent elements of $H$. This gives a version of Loday's dihedral homology, the specific signs involved giving the ``skew'' version.
\end{proof}

\begin{remark}
For the specific case of the tensor algebra $H=T(V)$, a theorem of Loday \cite[Thm. 3.1.6]{Loday} implies that $HD_n'(T(V))=HD_n'(\F)$ for $n\geq 1$.\footnote{Loday's statement in Theorem 3.1.6 that $HC_0(T(V))=S(V)$ is incorrect. It should be $T(V)/[T(V),T(V)]=\oplus V^{\otimes m}_{\mathbb Z_m}$, which follows from the main formula of that theorem.}
That is, the homology is carried by the subcomplex where all vertices are labeled by the unit element of $H$. In particular, the reduced dihedral homology vanishes for $n\geq 1$.
\end{remark}

\subsection{Graph homology of $\overline{\cGB_{H\Lie}^{(n)}}$ and $\overline{\cG_{H\Lie}^{(n)}}$}

\begin{theorem}
\label{thm:HReducedBased}
For $n\geq 2$ we have $H_k(\overline{\cGB_{H\Lie}^{(n)}}) = H^{2n-1-k}(\Aut(F_n);H^{\otimes n})$, where $\Aut(F_n)$ acts on $H^{\otimes n}$ via the representation $\rho_H$.
\end{theorem}

\begin{proof}
The based marked Lie graph complex is defined just like the based Lie graph complex, except that the graphs have markings, which are homotopy equivalences $\rho\colon G\to R_n$. Here $R_n$ is the wedge of $n$ circles. Let $\mathrm{BMLG}_{n,k}$ be the part of the marked Lie graph complex consisting of marked Lie graphs, with $k$ elements of the operad  $\Lie$ decorating the graph. Note that $\Aut(F_n)$ acts on this chain complex by changing the marking.

Let $\overline{B}\Aut(F_n)$ denote the moduli space of basepointed graphs, which is a rational classifying space for $\Aut(F_n)$. This is a quotient by $\Aut(F_n)$ of the space of marked basepointed graphs $\overline{E}\Aut(F_n)$ (auter space), which is contractible.

This chain complex $\mathrm{BMLG}_{n,\cdot}$ is constructed from $\overline{E}\Aut(F_n)$, giving a model for group cohomology with
$$C^{2n-1-k}(\overline{E}\Aut(F_n))\cong\mathrm{BMLG}_{n,k}.$$
So by definition we have $H_k([\mathrm{BMLG}_{n,\cdot}]_{\Aut(F_n)})=H^{2n-1-k}(\Aut(F_n);\Q)$.
Taking coefficients in the module $H^{\otimes n}$, we have
$$
H_k(\mathrm{BMLG}_{n,\cdot}\otimes_{\Aut(F_n)}H^{\otimes n})=H^{2n-1-k}(\Aut(F_n);H^{\otimes n}).
$$

Let $\alpha$ denote the map from $\overline{\cGB_{H\Lie}^{(n)}}$ to
$\mathrm{BMLG}_{n,\bullet}\otimes_{\Aut(F_n)} H^{\otimes n}$ constructed as follows:
for $\mathbf{G} \in\cG_{H\Lie}^{(n)}$ choose a maximal tree $\mathbf{T}$ of the underlying graph,
using the sliding relations and the commuting relations of elements in $H$ with Lie operad elements, we can push the elements of $H$ away from the edges in $\mathbf{T}$ and the half edge corresponding to the base point.
This allow us the represent the element $\mathbf{G}$ as an element in the $\cGB_\Lie^{(n)}$ decorated with elements in $H$ on the $n$ edges  which are outside the maximal tree $\mathbf{T}$. This element together with the ordering and orienting the edges outside $\mathbf{T}$ is an element in $\mathrm{BMLG}_{n}$, where the marking is given by collapsing the tree $\mathbf{T}$ to the base point of the rose $R_n$ and sending  additional edges edges to the loops in the rose $R_n$. I.e., label the edges not in the maximal tree by the generators of $F_n$ in some way. Next, we obtain an element in $\mathrm{BMLG}_{n}\otimes_{\Aut(F_n)} H^{\otimes n}$ by attaching an element in $H^{\otimes n}$ which is obtained by tensoring the elements from $H$ at the additional edges. This is illustrated in the following figure. On the left, the tree $\mathbf T$ is depicted in bolder lines. Elements of $H$ have been pushed onto the two edges not in the tree. This maps to a marked base pointed Lie graph with coefficients in $H^{\otimes 2}$ as on the right.
$$
\begin{minipage}{6cm}
\begin{tikzpicture}
\node[operad](A){$o_1$};
\node[empty](B)[above of =A]{};
\node[empty](C)[below of =A, node distance=1.4cm]{};
\node[empty](D)[right of =B]{};
\node[operad](K)[right of =D]{$o_2$};
\node[empty](J)[above of =K]{$\flat$};
\node[hopf](L)[right of =K, node distance=1.5cm]{$h_1$};
\node[empty](E)[right of = C]{};
\node[empty](F)[right of =E]{};
\node[hopf](G)[right of=F, node distance=1.5cm]{$h_2$};
\node[empty](H)[right of=L, node distance=1.5cm]{};
\node[operad](I)[below of=H]{$o_3$};
\begin{scope}[decoration={markings,mark = at position 0.5 with {\arrow{stealth}}}]
\draw[very thick, postaction=decorate]  (A) to[out=90, in=180](K);
\draw[very thick, postaction=decorate] (A) to (I);
\draw[very thick] (A) to[out=270, in=180] (G);
\draw[postaction=decorate]  (G) to[out=0,in=270](I);
\draw[very thick] (K) to (L);
\draw[postaction=decorate] (L) to[out=0, in=90]  (I);
\end{scope}
\draw[very thick] (K) to (J);
\end{tikzpicture}
\end{minipage}
\mapsto\,\,
\begin{minipage}{6cm}
\begin{tikzpicture}
\node[operad](A){$o_1$};
\node[empty](B)[above of =A]{};
\node[empty](C)[below of =A, node distance=1.4cm]{};
\node[empty](D)[right of =B]{};
\node[operad](K)[right of =D]{$o_2$};
\node[empty](J)[above of =K]{$\flat$};
\node[empty](L)[right of =K, node distance=1.5cm]{};
\node[empty](E)[right of = C]{};
\node[empty](F)[right of =E]{};
\node[empty](G)[right of=F, node distance=1.5cm]{};
\node[empty](H)[right of=L, node distance=1.5cm]{};
\node[operad](I)[below of=H]{$o_3$};
\node[empty](aa)[above of =G, node distance=.25cm]{$x_2$};
\node[empty](bb)[above of =L, node distance=.25cm]{$x_1$};
\begin{scope}[decoration={markings,mark = at position 0.5 with {\arrow{stealth}}}]
\draw[very thick, postaction=decorate] (A) to[out=90, in=180] (K);
\draw[very thick, postaction=decorate] (A) to (I);
\draw (A) to[out=270, in=180] (G.center);
\draw[postaction=decorate]  (G.center) to[out=0,in=270] (I);
\draw (K) to (L.center);
\draw[postaction=decorate]  (L.center) to[out=0, in=90] (I);
\end{scope}
\draw[very thick] (K) to (J);
\end{tikzpicture}
\end{minipage}
\otimes (h_1\otimes h_2)
$$
It is easy to see that the action of $\Aut(F_n)$ on $H^{\otimes n}$ is such that the resulting element does not depend on the choice of the ordering and orienting the external edges.

The main property one needs to verify is that  $\alpha(\mathbf{G})$ does not depend on the choice of the maximal tree $\mathbf{T}$. In fact this was our main motivation in the definition of the action of $\Aut(F_n)$ on $H^{\otimes n}$. First we consider the following elementary move that changes the maximal tree by one edge. Here the black line represents part of the tree $\mathbf{T}$ and the dashed edges are exterior to it.
$$
\begin{minipage}{3cm}
\begin{tikzpicture}
\node[empty](A){};
\node[operad](B)[right of=A, node distance=1.5cm]{$o$};
\node[empty](C)[right of =B, node distance=1.5cm]{};
\node[empty] (D)[above of=C]{};
\node[empty](E)[below of = C]{};
\node[empty](A1)[above of=A]{};
\node[empty](A2)[below of =A]{};
\begin{scope}[decoration={markings,mark = at position 0.5 with {\arrow{stealth}}}]
\draw[very thick] (A) to[postaction=decorate] (B);
\draw[densely dashed, postaction=decorate] (B) to(C);
\draw[densely dashed,postaction=decorate ] (B) to[out =45, in =180](D);
\draw[densely dashed,postaction=decorate ] (B) to[out=-45, in=180](E);
\end{scope}
\end{tikzpicture}
\end{minipage}
\mapsto
\begin{minipage}{3cm}
\begin{tikzpicture}
\node[empty](A){};
\node[operad](B)[right of=A]{$o$};
\node[empty](C)[right of =B, node distance=1.5cm]{};
\node[empty] (D)[above of=C]{};
\node[empty](E)[below of = C]{};
\begin{scope}[decoration={markings,mark = at position 0.5 with {\arrow{stealth}}}]
\draw[dashed, postaction=decorate] (A) to (B);
\draw[dashed, postaction=decorate] (B) to(C);
\draw[very thick, postaction=decorate]  (B) to[out =45, in =180](D);
\draw[dashed, postaction=decorate] (B) to[out=-45, in=180](E);
\end{scope}
\end{tikzpicture}
\end{minipage}
$$
This has the following effect on markings:
$$
\begin{minipage}{3cm}
\begin{tikzpicture}
\node[empty](A){};
\node[operad](B)[right of=A]{$o$};
\node[empty](C)[right of =B, node distance=1.5cm]{};
\node[empty] (D)[above of=C]{};
\node[empty](E)[below of = C]{};
\begin{scope}[decoration={markings,mark = at position 0.5 with {\arrow{stealth}}}]
\draw[very thick, postaction=decorate]  (A) to(B);
\draw[densely dashed, postaction=decorate] (B) to(C);
\draw[densely dashed, postaction=decorate] (B) to[out =45, in =180](D);
\draw[densely dashed, postaction=decorate] (B) to[out=-45, in=180](E);
\end{scope}
\node[empty](dd) [above of=D, node distance=.25cm]{$x_1$};
\node[empty](cc) [above of=C, node distance=.25cm]{$x_2$};
\node[empty](ee) [above of=E, node distance=.25cm]{$x_3$};
\end{tikzpicture}
\end{minipage}
\mapsto
\begin{minipage}{3cm}
\begin{tikzpicture}
\node[empty](A){};
\node[operad](B)[right of=A]{$o$};
\node[empty](C)[right of =B, node distance=1.5cm]{};
\node[empty] (D)[above of=C]{};
\node[empty](E)[below of = C]{};
\begin{scope}[decoration={markings,mark = at position 0.5 with {\arrow{stealth}}}]
\draw[densely dashed, postaction=decorate] (A) to (B);
\draw[densely dashed, postaction=decorate] (B) to(C);
\draw[very thick, postaction=decorate] (B) to[out =45, in =180](D);
\draw[densely dashed, postaction=decorate] (B) to[out=-45, in=180](E);
\end{scope}
\node[empty](dd) [above of=D, node distance=.25cm]{};
\node[empty](cc) [above of=C, node distance=.3cm]{$x_1^{-1}x_2$};
\node[empty](ee) [above of=E, node distance=.3cm]{$x_1^{-1}x_3$};
\node[empty](aa) [above of =A, node distance=.25cm]{\,\,$x_1$};
\end{tikzpicture}
\end{minipage}
$$
It has the following effect on the elements of $H$ labeling the edges, where $\Delta^3(h_1)=h_1'\otimes h_1''\otimes h_1'''$ in Sweedler notation:
$$
\begin{minipage}{4cm}
\begin{tikzpicture}
\node[empty](A){};
\node[operad](B)[right of=A]{$o$};
\node[hopf](C)[right of =B, node distance=1.5cm]{$h_2$};
\node[hopf] (D)[above of=C]{$h_1$};
\node[hopf](E)[below of = C]{$h_3$};
\node[empty] (G)[right of =C]{};
\node[empty] (H)[right of =D]{};
\node[empty] (I)[right of =E]{};
\begin{scope}[decoration={markings,mark = at position 0.5 with {\arrow{stealth}}}]
\draw[very thick, postaction=decorate] (A) to (B);
\draw[densely dashed, postaction=decorate] (B) to(C);
\draw[densely dashed, postaction=decorate] (B) to[out =45, in =180](D);
\draw[densely dashed, postaction=decorate] (B) to[out=-45, in=180](E);
\end{scope}
\draw[densely dashed]  (C) to(G);
\draw[densely dashed]  (D) to (H);
\draw[densely dashed]  (E) to (I);
\end{tikzpicture}
\end{minipage}
\mapsto\,\,
\begin{minipage}{6cm}
\begin{tikzpicture}
\node[hopf](A){$h_1'$};
\node[operad](B)[right of=A, node distance=1.5cm]{$o$};
\node[hopf](C)[right of =B, node distance=2cm]{$S(h_1'')h_2$};
\node[empty] (D)[above of=C]{};
\node[hopf](E)[below of = C]{$S(h_1''')h_3$};
\node[empty] (aa)[left of =A]{};
\node[empty](cc) [right of =C, node distance =2cm]{};
\node[empty](dd) [right of =D, node distance =2cm]{};
\node[empty](ee) [right of =E, node distance =2cm]{};
\begin{scope}[decoration={markings,mark = at position 0.5 with {\arrow{stealth}}}]
\draw[densely dashed] (A) to (B);
\draw[densely dashed, postaction=decorate] (B) to(C);
\draw[thick, postaction=decorate]  (B) to[out =45, in =180](D);
\draw[densely dashed, postaction=decorate] (B) to[out=-45, in=180](E);
\draw[densely dashed] (aa) to[postaction=decorate] (A);
\draw[densely dashed] (C) to (cc);
\draw[very thick] (D) to (dd);
\draw[densely dashed] (E) to (ee);
\end{scope}
\end{tikzpicture}
\end{minipage}
$$
The automorphism $\varphi: x_1\mapsto x_1, x_2\mapsto x_1^{-1}x_2, x_3\mapsto x_1^{-1}x_3$ takes $h_1\otimes h_2\otimes h_3$ to $h_1'\otimes S(h_1'')h_2\otimes S(h_1''')h_3$, so both maximal trees give the same element of $\mathrm{BMLG}_{n}\otimes_{\Aut(F_n)} H^{\otimes n}$.

Consider a generalization of the above move where many edges of $\mathbf{T}$ meet the vertex in question. The effect of changing an edge of $\mathbf{T}$ to a different edge emanating from the vertex will have a similar effect on the marking, except that some other edges elsewhere in the graph will have an $x_1$ multiplied on their right or an $x_1^{-1}$ multiplied on their left. This change will be mirrored in the change in the labeling by elements of $H$. It is not difficult so see that this generalized move is sufficient to move between any two maximal trees.

From the construction it is clear that  $\alpha$ is both surjective and injective. It is also a chain map -- collapsing an edge from $T$ clearly
commutes with $\alpha$. Therefore $\alpha$ induces an isomorphism on the homology.
\end{proof}
\begin{remark}
 We actually discovered the action of $\Aut(F_n)$ on $H^{\otimes n}$
by trying to express $H_*(\cGB_{H\Lie}^{(n)})$ as the top cohomology of $\Aut(F_n)$ with coefficients in some module.
\end{remark}

\begin{theorem}
\label{thm:HReduced}
For $n\geq 2$ we have $H_k(\overline{\cG_{H\Lie}^{(n)}}) = H^{2n-2-k}(\Out(F_n);\overline{H^{\otimes n}})$.
\end{theorem}
\begin{proof}
The proof is almost the same as of Theorem~\ref{thm:HReducedBased}. The only differences are: (i) one uses the \emph{marked Lie graph complex} instead of \emph{based marked Lie graph complex}; (ii) after one pushes the $H$ elements outside the maximal tree $\mathbf{T}$ the $H$-labels on the remaining $n$ edges are not uniquely determined, because there is no unique direction to push away from. This ambiguity corresponds to the fact that one can change the group elements marking each edge by a global conjugation, and by definition we mod out by such conjugations in $\overline{H^{\otimes n}}$.
\end{proof}

\subsection{Reduction from $\cGB_{H\Lie}^{(n)}$ to $\overline{\cGB_{H\Lie}^{(n)}}$ }

\begin{theorem}
The induced map $H_\bullet(\cGB_{H\Lie}^{(n)}) \to H_\bullet(\overline{\cGB_{H\Lie}^{(n)}})$ is an isomorphism.
\end{theorem}
\begin{proof}
This theorem essentially allows us to remove the bivalent vertices from the underlying graphs.
First we divide the edges in the graph into two types: red, where at least one end point is a bivalent vertex and blue, which connect two vertices of higher degree (the degree of the the based vertex counts also the base hair). This allow us to consider $\cGB_{H\Lie}^{(n)}$ as a double complex, whose homology can be computed using spectral sequence arguments. Our aim is to show that the spectral sequence collapses since the vertical complexes are exact except at degree $0$.

Each vertical complex breaks as a sum over graphs without bivalent vertices, and over each graph $G$ the complex consists of colored graphs obtained by adding chains of bivalent vertices at the edges of $G$. First we order the edges of $G$ so for each $k< |V|$ the first $k$ edges span a tree which contain the base vertex. This allow us to think of the complex over the graph $G$ as $|E|$-dimensional complex (multi graded by the number of vertices on each edge). Using induction we will shows that all resulting spectral sequences collapse and that the homology of the complex is only at degree $0$.

\begin{claim}
\label{cl:bar-resolution}
For each $l \leq |E|$ the homology of the $l$-dimensional subcomplex is only at degree $0$.
If $l < |V|$ in the homology one can push all $H$ elements away from the first $l$ edges of the graph.
For $|V| \leq l \leq |E|$ in the homology one can push all $H$ elements away from the first $|V|-1$ edges of the graph, but there needs to be an $H$ element on each of the remaining $l-|V| +1$ edges.
\end{claim}
\begin{proof}
Each complex correspond to a bar resolution of the algebra $H$ where at the ends one has operad elements in $H\Lie$ viewed as $H$- modules where $H$ acts by composing along the corresponding input. This gives that the homology of the complex computes the $\Tor^H_\bullet$-functors of $H$-modules at the end points.

The key observation is that these modules are free $H$ modules therefore all but the first $\Tor$ functor are trivial, because each $H\Lie\arity{n}$ is a free $H$ module when $H$ acts by composition on any i/o slot (this is true even if $n-1$ copies of $H$ act on $n-1$ i/o slots). The $\Tor_0$ is just the tensor product of these modules which is equivalent of allowing the slides of $H$ elements across the corresponding edge.

For the second part of the claim one uses the commutation relation between the element of $H$ on the elements of the operad $\Lie$ which allows one to push the $H$ elements away from the base vertex.
\end{proof}

The claim implies that the homology of the vertical complex above the graph $G$ collapses, which lead to a collapse of the spectral sequence.
\end{proof}

Unfortunately this argument does not work for the graph complex since it ``does not have enough edges'' --
the statement of Claim~\ref{cl:bar-resolution} is valid for $l< |E|$ but not for $l=|E|$ because we obtain a bar resolution where the two modules are not free. As a result the higher $\Tor$s might not vanish.

However in degree $1$ a similar result holds and the proof is even easier.
\begin{theorem}
The induced map $H_1(\cG_{H\Lie}^{(n)}) \to H_1(\overline{\cG_{H\Lie}^{(n)}})$ is an isomorphism.
\end{theorem}
\begin{proof}
Observe that the kernel of the map $\cG_{H\Lie,1}^{(n)} \longrightarrow \overline{\cG_{H\Lie,1}^{(n)}}$ lies in the image of the differential.
\end{proof}

\section{Application to the Johnson cokernel}\label{sec:cokernel}
In this section we review the definition of the Johnson homomorphism and its cokernel, as well as the generalized trace map from~\cite{C}. Finally we draw the connection with $H_1(\mathcal G_{T(V)\Lie})$ and $H^{2r-3}(\Out(F_r); \overline{T(V)^{\otimes r}})$.

As before, let $\F$ be a field of characteristic $0$. Let $\Sigma_{g,1}$ be a surface of genus $g$ with one boundary component. It has free fundamental group generated by embedded curves $x_1,\ldots, x_g, y_1,\ldots y_g$ with $x_i,y_i$ intersecting in one point, and all other intersections trivial.
Let $V=H_1(\Sigma_{g,1};\F)$, which is a symplectic vector space with symplectic form $\la\cdot,\cdot\ra$, and let $p_1,\ldots,p_g,q_1,\ldots, q_g$ be the symplectic basis which is the image of the generating set of the fundamental group. We say $\la v,w\ra$ is the \emph{contraction} of $v$ and $w$.
For the groups $G\in\{\SP(V),\GL(V),\sym{s}\}$, let $[\lambda]_{G}$ be the irreducible representation of $G$ corresponding to $\lambda$.

We begin by defining the relevant Lie algebra which is the target of the Johnson homomorphism.

\begin{definition}
 Let $\mathsf{L}_k(V)$ be the degree $k$ part of the free Lie algebra on $V$. Define $\mathsf D_s(H)$ to be the kernel of the bracketing map $V\otimes \mathsf  L_{s+1}(V)\to \mathsf L_{s+2}(V)$. Let $\mathsf D(V)=\bigoplus_{s=0}^\infty \mathsf D_s(V)$ and $\mathsf D^+(V)=\bigoplus_{s\geq 1} \mathsf D_s(V)$. We refer to $s$ as the \emph{order} of an element of $\mathsf D(V)$.
\end{definition}
$V\otimes \mathsf L(V)$ is canonically isomorphic via the symplectic form to $V^*\otimes \mathsf L(V)$ which is isomorphic to the space of derivations $\mathsf {Der} (\mathsf L(V))$. Under this identification, the subspace $\mathsf D(V)$ is identified with $\mathsf {Der}_\omega(\mathsf L(V))=\{X\in \mathsf {Der}(V)\,|\, X\omega =0\}$ where $\omega=\sum [p_i,q_i]$. Thus $\mathsf D(V)$ is a Lie algebra with bracket coming from $\mathsf {Der}_\omega(V)$.

There is another beautiful interpretation of this Lie algebra in terms of trees:
 \begin{definition}
  Let $\cT(V)$ be the vector space of unitrivalent trees where the univalent vertices are labeled by elements of $V$ and the trivalent vertices each have a specified cyclic order of incident half-edges, modulo the  standard AS, IHX and multilinearity relations.(See Figure~\ref{fig:multilin} for the multilinearity relation.) Let $\cT_k(V)$ be the part with $k$ trivalent vertices. Define a Lie bracket on $\cT(V)$ as follows. Given two labeled trees $t_1,t_2$, the bracket $[t_1,t_2]$ is defined by summing over joining a univalent vertex from $t_1$ to one from $t_2$, multiplying by the contraction of the labels.
\end{definition}
These two spaces $\mathsf{D}_s(V)$ and $\cT_s(V)$ are connected by a map $\eta_s\colon \cT_s(V)\to V\otimes \mathsf L_{s+1}(V)$ defined
by $\eta_s(t)=\sum_x \ell(x)\otimes t_x$ where the sum runs over univalent vertices $x$, $\ell(x)\in V$ is the label of $x$, and $t_x$ is the element of $\mathsf L_{s+1}(V)$ represented by the labeled rooted tree formed by removing the label from $x$ and regarding $x$ as the root. The image of $\eta_s$ is contained in $\mathsf D_s(V)$ and gives an isomorphism $\cT_s(V)\to \mathsf D_s(V)$ in this characteristic $0$ case~\cite{Levine}.

\begin{figure}
$$
\begin{minipage}{2cm}
\includegraphics[width=2cm]{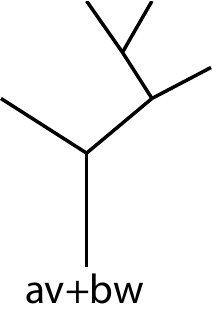}
\end{minipage}
=
a
\begin{minipage}{2cm}
\includegraphics[width=2cm]{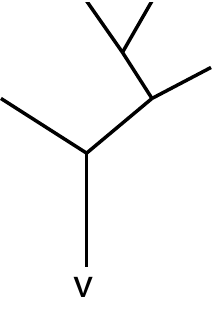}
\end{minipage}
+b
\begin{minipage}{2.7cm}
\includegraphics[width=2.7cm]{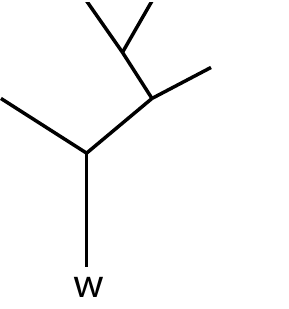}
\end{minipage}
$$
\caption{Multilinearity relation in $\cT(H)$. Here $a,b\in\F$, $v,w\in V$}\label{fig:multilin}
\end{figure}


Now that we understand the target of the Johnson homomorphism, we review the construction of the homomorphism itself.
Let $F=\pi_1(\Sigma_{g,1})$ be a free group on $2g$ generators and given a group $G$, let $G_k$ denote the $k$th term of the lower central series: $G_1=G$ and $G_{k+1}=[G,G_k]$. The Johnson filtration
$$
\Mod(g,1)= \mathbb J_0 \supset \mathbb J_1\supset \mathbb J_2 \supset \cdots
$$
of the mapping class group $\Mod(g,1)$ is defined by letting $\mathbb J_s$ be the kernel of the homomorphism $\Mod(g,1)\to \Aut(F/F_{s+1})$. The \emph{associated graded} $\mathsf J_s$ is defined by $\mathsf J_s=\mathbb J_s/\mathbb J_{s+1}\otimes \F$. (The Johnson filtration is a central series, so that the groups $\mathsf J_s$ are abelian.) Let $\mathsf J=\bigoplus_{s\geq 1} \mathsf J_s$, where we refer to $s$ as the \emph{order}\footnote{Some authors use ``level" instead of ``order."} of the element.

The group commutator on $\Mod(g,1)$ induces a Lie algebra structure on $\mathsf J$.

It is well-known that $\Mod(g,1)\cong \Aut_0(F)$ where $$\Aut_0(F)=\left\{\varphi\in \Aut(F)\,\left|\, \varphi(\prod_{i=1}^g[x_i,y_i])=\prod_{i=1}^g[x_i,y_i]\right. \right\}.$$

\begin{definition}
 The total Johnson homomorphism $\tau\colon \mathsf J\to \mathsf D^+(V)$ is defined as follows. Let $\varphi\in \mathbb J_s$. Then $\varphi$ induces the identity on $\Aut(F/F_{s+1})$. Hence for every $z\in F$, $z^{-1}\varphi(z)\in F_{s+1}$, and we can project to get an element $[z^{-1}\varphi(z)]\in F_{s+1}/F_{s+2}\otimes \F\cong \mathsf L_{s+1}(H)$.
 Define a map $\tau(\varphi)\colon V\to  \mathsf L_{s+1}(V)$ via $z\mapsto [z^{-1}\varphi(z)]$ where $z$ runs over the standard symplectic basis of $V$. By the various identifications, we can regard $\tau(\varphi)$ as being in $V \otimes \mathsf L_{s+1}(V)$.
 The fact that $\varphi$ preserves $\prod_{i=1}^g [x_i,y_i]$ ensures that $\tau(\varphi)\in\mathsf{D}_s(V)\subset V \otimes \mathsf L_{s+1}(V)$.
 \end{definition}

\begin{proposition}[Morita]
The Johnson homomorphism $\tau \colon \mathsf J\to \mathsf D^+(V)$ is an injective homomorphism of Lie algebras.
\end{proposition}

The main object of study for the remainder of this paper is \emph{the Johnson cokernel:}
$$
\mathsf C_s=\mathsf D_s(V)/\tau(\mathsf J_s).
$$
More precisely, we are interested in the stable part of the cokernel and we always assume that $2g=\dim (V)\gg s$.

This is facilitated by a theorem of Hain~\cite{Ha}:
\begin{theorem}[Hain]
In the stable range, the image of the Johnson homomrphism
$\tau$ is the Lie subalgebra of $\mathsf D^+(V)$ generated by $\mathsf D_1(V) = \tau(\mathsf J_1)\cong \ext^3 V$.
\end{theorem}

Let $\mathcal S_2\subset \mathcal G_{T(V)\Lie,2}$ be the $\F$ vector space spanned by graphs where one of the two vertices is colored by an element of $\Lie\arity{3}\subset \Lie_V\arity{3}$, or by elements of $\Lie_V\arity{3}$ of degree $1$.
There are three types of graphs that result:
all three i/o slots of a tripod are connected by edges to the rest of the graph;
two of the slots can be connected to each other by an edge, with the other connecting to the rest of the graph by an edge;
or finally, one i/o slot can be labeled by a vector from $V$, while the other two i/o slots are connected to the rest of the graph by edges.

Define $\Omega_r(V)=  \mathcal G^{(r)}_{T(V)\Lie,1}/\partial \mathcal S_2$ for $r\geq 1$.
In~\cite{C}, a generalized trace map

$$
\Tr^C\colon \mathsf C\to \bigoplus_{r\geq 1}\Omega_r(V)
$$
is constructed, which moreover {stably} surjects onto at least the $\SP$-modules $[\lambda]_{\SP}\subset [\lambda]_{\GL}\subset \Omega_r(V)$.

\begin{remark}
In \cite{C}, the target of the trace map is actually a quotient of the first chain group of hairy Lie graph homology: $\mathcal H\Lie\cong \mathcal G_{\Lie_V}$ (Proposition~\ref{prop:hairyG}).
Now $\mathcal G^{(r)}_{\Lie_V,1}\cong \mathcal G^{(r)}_{\overline{\Lie_V},1}$ in ranks $r\geq 1$, since the existence of at least one edge joining the element of $\Lie_V$ to itself makes the arity $\geq 1$. Since $\overline{\Lie_V}\cong T(V)\Lie$ (Proposition~\ref{prop:Lie-iso}), this implies the target $\Tr^C$ can be defined as a quotient of $\mathcal G^{(r)}_{T(V)\Lie,1}$ as stated above.
\end{remark}

The trace map is induced by maps $\Tr^{(r)} \colon \mathcal T(V)\to \mathcal G^{(r)}_{T(V)\Lie,1}$ which are defined by summing over all ways of adding $r$ external edges to a tree, multiplying by the product of the contractions of the endpoint labels. In~\cite{C}, it is shown that $\Tr(\im\tau)\subset \partial \mathcal S_2$, so induces a map from the cokernel to $\Omega_r(V)$.

\begin{proposition}
There is a surjection
$$
\Omega_r(V)\twoheadrightarrow H_1(\mathcal G^{(r)}_{T(V)\Lie}).
$$
Thus there are maps
$\mathsf C_{k +2r}\to  H_1(\mathcal G^{(r)}_{T(V)\Lie})_k$ for $k\geq 1$, which {stably} surject onto at least the $\SP$-modules $[\lambda]_{\SP}\subset [\lambda]_{\GL}$ in the target.
\end{proposition}
\begin{proof}
This follows because $\partial \mathcal S_2\subset \partial \mathcal G^{(r)}_{T(V)\Lie,2}$.
\end{proof}

Recall that
$$
H_1(\mathcal G^{(r)}_{H\Lie})\cong
\begin{cases}
[\overline{H^{\otimes 1}}]_{\mathbb Z_2}& r=1\\
H^{2r-3}(\Out(F_r);\overline{H^{\otimes r}})&r>1.
\end{cases}
$$

\begin{proposition}\
\begin{enumerate}
\item The composition $\mathsf C_{k+2} \to \left[\overline{T(V)^{\otimes 1}}_{\Z_2}\right]_k\cong \left[V^{\otimes k}\right]_{D_{2k}}$ recovers the Enomoto-Satoh trace invariant when $k\geq 2$.
\item The composition $\mathsf C\to H_1(\mathcal G^{(r)}_{T(V)\Lie})\twoheadrightarrow  H_1(\mathcal G^{(r)}_{\Sym(V)\Lie})$ recovers the CKV trace $\Tr^{CKV}$, which detects the abelianization $\mathsf D^+_{\mathrm{ab}}(V)$.
\end{enumerate}
\end{proposition}
\begin{proof}
See ~\cite{C}.
\end{proof}

\section{Some rank 2 computations}\label{sec:calculations}
The previous section showed that $H^1(\Out(F_2);\overline{T(V)^{\otimes 2}})$ is the target of an invariant of the Johnson cokernel, so we would like to calculate this cohomology. This turns out to be a difficult question in general, but we make partial progress in this section by considering the quotient of $T(V)\cong U(\mathsf L(V))$ by commutators of length $3$: i.e. $U(\mathsf L_{(2)}(V))$, where $\mathsf{L}_{(k)}=\mathsf{L}(V)\left/\left(\oplus_{i\geq {k+1}}\mathsf{L}_i(V)\right)\right.$ is the free nilpotent Lie algebra of nilpotency class $k$.  (In what follows, we will sometimes suppress the dependence of $\mathsf L(V)$ and $\mathsf L_{(2)}(V)$ on $V$.)

It follows from Theorem~\ref{thm:class2} that
$\overline{\Sym(\sL_{(2)}\otimes\F^2)}$ and  $\overline{U(\sL_{(2)})^{\otimes 2}}$ are isomorphic as $\SL^{\pm}_2(\F)\times \GL(V)$-modules. From this we deduce the following corollary.

\begin{corollary}
There is a surjection
$$
H^1(\Out(F_2);\overline{T(V)^{\otimes 2}})\twoheadrightarrow
H^1(\GL_2(\Z);\overline{\Sym(\sL_{(2)}\otimes \F^2)})
$$
where $\GL_2(\Z)$ acts on $\overline{\Sym(\mathsf{L}_{(2)}\otimes \F^2)}$ by acting via the standard action on $\F^2$.
\end{corollary}
\begin{proof}
The surjection follows from the general statement that if $\mathfrak{g} \twoheadrightarrow \mathfrak{h}$ is surjective it induces a surjection
$H^{vcd}(\Out(F_r); \overline{U(\mathfrak{g})^{\otimes r}}) \twoheadrightarrow H^{vcd}(\Out(F_r); \overline{U(\mathfrak{h})^{\otimes r}}) $. The easiest way to see that is on the level of
$H_1(\mathcal G^{(r)}_{H\Lie})$ since there is a clear surjection $\mathcal G^{(r)}_{U(\mathfrak g)\Lie,1}\twoheadrightarrow \mathcal G^{(r)}_{U(\mathfrak h)\Lie,1}$, which remains surjective upon dividing by boundaries. 
\end{proof}

Because the action of $\GL_2(\Z)$ in the previous corollary is a standard action, the cohomology is amenable to attack using classical theory. This computation is given in the following theorem.

\begin{theorem}\label{thm:L2calc}
The cohomology group $H^1(\GL_2(\Z);\overline{\Sym(\mathsf{L}_{(2)}\otimes \F^2)})$ is isomorphic to
$$
\bigoplus_{k>\ell\geq 0} \left(\mathcal S_{2k-2\ell+2}\otimes \overline{\SF{(2k,2\ell)}}(\sL_{(2)}) \right)\oplus
\bigoplus_{k>\ell\geq 0} \left(\mathcal M_{2k-2\ell+2}\otimes \overline{\SF{(2k+1,2\ell+1)}}(\sL_{(2)})\right)
$$
where $\mathcal M_s$ and $\mathcal S_s$ are spaces of weight $s$ modular forms and cusp forms respectively. Here $\overline{\SF{\lambda}}(\sL_{(2)})$ is the quotient the Schur functor $\SF{\lambda}(\sL_{(2)})$ by the image of the adjoint action.
\end{theorem}
\begin{proof}

We have
$$
H^1(\GL_2(\Z);\overline{\Sym(\sL_{(2)}\otimes \F^2)})=
\sum_{\lambda}  H^1(\GL_2(\Z);\SF{\lambda}(\F^2)) \otimes  \left( \SF{\lambda}(L) / \ad(\sL)\cdot ( \SF{\lambda}(\sL) ) \right).
$$
The first homology $H^1(\Out(F_2);\SF{\lambda}(\F^2))$ is computed in~\cite{CKV1}:
$$
H^1(\Out(F_2);\SF{(k,\ell)}(\F^2))=\begin{cases}
0& \text{if }k+\ell\text{ is odd}\\
\mathcal S_{k-\ell+2} &\text{if } k,\ell\text{ are both even} \\
\mathcal M_{k-\ell+2}&\text{if } k,\ell\text{ are both odd}
\end{cases}
$$
completing the proof.
\end{proof}

Notice that projecting $\sL_{(2)}\to \sL_{(1)}\cong V$ recovers the calculation from~\cite{CKV1}. The task now becomes to understand the modules $\overline{\SF{\lambda}(\sL_{(2)})}$ as direct sums of irreducible $\GL(V)$-representations.

We begin by stating a well-known lemma. See e.g. \cite[Example 6.11]{FH}.
\begin{lemma}\label{lem:schurdecomp}
For any two vector spaces $V$ and $W$ we have
$$
\SF{\lambda}(V\oplus W) =
\bigoplus c_{\mu,\nu}^{\lambda} \SF{\mu}(V) \otimes \SF{\nu}(W),
$$
where $c_{\mu,\nu}^{\lambda}$ denotes the multiplicity of
$P_\lambda$ in the representation $P_\mu \circ P_\nu = \mathrm{Ind}^{\sym{|\lambda|}}_{\sym{|\mu|} \times \sym{|\nu|}}$, which can be computed using the Littlewood Richardson rule.
\end{lemma}

\begin{proposition}
\label{prop:onepart}
$\SF{(k)}(\ext^2 V)=\bigoplus_{\lambda\in \mathcal A_k} \SF{\lambda}(V),$ where $\mathcal A_k$ is the set of Young diagrams with $2k$ boxes and an even number of boxes in each column.
\end{proposition}
\begin{proof}
See  \cite[Example I.8.6]{MacDonald} or   \cite[Proposition 2.3.8]{Weyman}.
\end{proof}

\begin{remark}
We have the following algorithm for calculating $\SF{\lambda}(\ext^2 V)$. When $\lambda$ has only one part this is Proposition~\ref{prop:onepart}.
For $\lambda$ consisting of two parts this can be does using a few tensor products
(using that $\SF{(p,q)}(W)$ is the kernel of a surjective map from
$\SF{(p)}(W) \otimes \SF{(q)}(W)$ to $\SF{(p+1)}(W) \otimes \SF{(q-1)}(W)$
which can be done by using the LR rule).\end{remark}

Together with Lemma~\ref{lem:schurdecomp}, this leads to a method for computing $\SF{\lambda}(\sL_{(2)})$.

\begin{example}
We illustrate the method by calculating $\SF{(2,1)}(V\oplus \ext^2 V)$. The partition $[21]$ appears with multiplicity $1$ when multiplying $[1]$ and $[2]$, $[1]$ and $[1^2]$ or $[21]$ and $[0]$. So we get
\begin{align*}
\SF{(2,1)}(V\oplus \ext^2 V)= &\SF{(1)}(V)\otimes\SF{(2)}(\ext^2 V)\oplus  \SF{(2)}(V)\otimes\SF{(1)}(\ext^2 V)\oplus
\SF{(1)}(V)\otimes\SF{(1^2)}(\ext^2 V)\\ &\oplus \SF{(1^2)}(V)\otimes\SF{(1)}(\ext^2 V)\oplus \SF{(2,1)}(V)\oplus \SF{(2,1)}(\ext^2 V)
\end{align*}
By Lemma~\ref{lem:schurdecomp}, we have $\SF{(2)}(\ext^2 V)\cong [1^4]\oplus [2^2]$. We compute $\SF{[1^2]}(\ext^2V)$ as the kernel of the map
$\ext^2 V\otimes \ext^2 V\to S^2(\ext^2 V)$. By the LR rule, this is the kernel of the map $[21^2]\oplus [2^2]\oplus [1^4]\to [1^4]\oplus [2^2]$. Hence
$\SF{[1^2]}(\ext^2V)\cong [21^2]$.  Finally, we compute $\SF{(2,1)}(\ext^2 V)$ as the kernel of the map
$$S^2\ext^2 V\otimes \ext^2V\to S^3(\ext^2 V),$$
which is the same as $$([1^4]\oplus [2^2])\otimes [1^2]\to [1^6]\oplus [2^21^2]\oplus[3^2].$$
By LR, we have $$([1^4]\oplus [2^2])\otimes [1^2]\cong [21^4]\oplus[2^21^2]\oplus [1^6]\oplus [2^21^2]\oplus [3^2]\oplus [321].$$
Hence $\SF{(2,1)}(\ext^2 V)=[321]\oplus[2^21^2]\oplus [21^4]$.
Putting these calculations together, we get
\begin{align*}
\SF{(2,1)}(V\oplus \ext^2 V)=&[1]\otimes ( [1^4]\oplus [2^2])\oplus ([2]\otimes [1^2]) \oplus ([1]\otimes [21^2])\\&\oplus ([1^2]\otimes[1^2])\oplus
[21]\oplus ([321]\oplus[2^21^2]\oplus [21^4])\\
=& [1^5]\oplus 2[21^3]\oplus[32]\oplus 2[2^21]\oplus [31]\oplus 2[21^2]\oplus [31^2]\oplus[1^4]\\
&\oplus[2^2]\oplus[21]\oplus[321]\oplus[2^21^2]\oplus[21^4]
\end{align*}
These computations can be efficiently carried out using Sage. After loading the relevant library using \texttt{SF = SymmetricFunctions(QQ)} and
defining \texttt{s = SF.schur()}, the above computation can be verified with the command \texttt{s([2,1])(s[1]+s([1,1]))}.

\end{example}

Now we need to compute the image of the adjoint action.

\begin{lemma}\label{lem:ad}
The map $\ad:V \otimes \SF{\lambda}(V) \to \SF{\lambda}(\sL_{(2)})$ induced by the Lie bracket is injective
for all partitions $\lambda$ with at least two parts; when $\lambda = (n)$ the kernel is isomorphic to
$\SF{(n+1)}(V)$.
\end{lemma}
\begin{proof}
First consider the case of $\lambda=(n)$. The image of the adjoint map lands in $[S^{n}(\sL_{(2)})]_{n+1}=\ext^2 V
{\otimes} S^{n-1}(V)=[n,1]_{\GL}\oplus[n-1,1^2]_{\GL}$, and we are dividing out by the image of the adjoint action $V\otimes S^{n-1}(V)\to S^{n}(\mathsf{L}_{(2)})$. The source has decomposition $[n+1]_{\GL}\oplus [n,1]_{\GL}$ and since the map is not zero, the $[n,1]_{\GL}$ component must map onto the corresponding component in the image, leaving $[n+1]$ as the kernel.

For general $\lambda$, embed $\SF{\lambda}(V)$ into $\bigotimes_i \Sym^{\lambda_i}(V)$. Then the image of $\ad$ on $V \otimes \bigotimes_i \Sym^{\lambda_i}(V)$ lands in $\bigoplus_{i} [S^{\lambda_i}(\sL_{(2)})]_{\lambda_i+1}\otimes \bigotimes_{j\neq i} S^{\lambda_i}(V)$. Letting $\pi_i$ be the projection to the $i$th summand of this direct sum, we see that kernel of $\ad$ on
$V \otimes \bigotimes_i \Sym^{\lambda_i}(V)$ is equal to
$$\bigcap_i \ker \pi_i\circ\ad=\bigcap_i [\lambda_1]\otimes\cdots\otimes[\lambda_i+1]\otimes\cdots[\lambda_{r}].$$
It is not too hard to see that
is just $\Sym^{|\lambda|+1}(V)$. 
\end{proof}

Note that the lowest degree  part of $\overline{\SF{\lambda}(\sL_{(2)})}$  is in degree $|\lambda|$ and is equal to $\SF{\lambda}(V)$. In the next lemma, we compute what happens in one degree above this for $\lambda=(p,q)$ which are what appear in Theorem \ref{thm:L2calc}.
In fact, we do not need parts $d$ and $e$, but we include them for completeness.
\begin{lemma}
a) For $p \geq 2$ the degree $p+1$ part of $\overline{\SF{(p)}(\sL_{(2)})}$ is given by
$$
\left[\overline{\SF{(p)}(\sL_{(2)})}\right]_{p+1}=\SF{(p-1,1^2)}(V).
$$

b) For $p\geq 3$ the degree $p+2$ part of $\overline{\SF{(p,1)}(\sL_{(2)})}$ is given by
$$
\left[\overline{\SF{(p,1)}(\sL_{(2)})}\right]_{p+2}=\SF{(p,1^2)}(V) \oplus \SF{(p-1,2,1)}(V) \oplus \SF{(p-1,1^3)}(V).
$$

c) For $p \geq q+2$ and $q>1$ the degree $p+q+1$ part of $\overline{\SF{(p,q)}(\sL_{(2)})}$ is given by
$$
\left[\overline{\SF{(p,q)}(\sL_{(2)})}\right]_{p+q+1}=
\SF{(p+1,q-1,1)}(V) \oplus
\SF{(p,q,1)}(V) \oplus
\SF{(p-1,q+1,1)}(V) \oplus
\SF{(p,q-1,1^2)}(V) \oplus
\SF{(p-1,q,1^2)}(V).
$$

d) For $p\geq 2$, the degree $2p+1$ part of $\overline{\SF{(p,p)}(\sL_{(2)})}$ is given by
$$\left[\overline{\SF{(p,p)}(\sL_{(2)})}\right]_{2p+1}=\SF{(p+1,p-1,1)}(V).$$

e) For $p\geq 2$, the degree $2p$ part of $\overline{\SF{(p,p-1)}(\sL_{(2)})}$ is given by
$$\left[\overline{\SF{(p,p-1)}(\sL_{(2)})}\right]_{2p}=\SF{(p+1,p-1)}(V)\oplus \SF{(p+1,p-2,1)}(V)\oplus \SF{(p,p-1,1)}(V).$$
\end{lemma}
\begin{proof}
We will do part c), since the other cases are similar.
First we compute
$$
\left[\SF{(p,q)}(\sL_{(2)})\right]_{p+q+1} = (\SF{(p,q-1)}(V) \oplus \SF{(p-1,q)}(V) ) \otimes \SF{(1,1)}(V).
$$
This is because you need $\nu=(1)$ in the LR decomposition to get degree $p+q+1$. Now this
is equal to (dropping the Schur functor notation)
\begin{multline*}
[p+1,q] + [p+1,q-1,1] + [p,q,1] + [p,q-1,1,1]
+\\
[p,q+1] + [p,q,1] + [p-1,q+1,1] + [p-1,q,1,1]
\end{multline*}
which needs to be factored by the image of $\ad$ which is just $V\otimes \SF{(p,q)}(V)$, i.e.,
$$
[p+1,q] + [p,q+1] + [p,q,1]
$$
the difference of these two is the expression in the statement of the lemma
\end{proof}
\begin{remark}
The same idea can be used to compute the degree $p+q+2$ part of $\overline{\SF{(p,q)}(\sL_{(2)})}$.
For example if $q=0$ the result will be $[p-2,2^2] + [p-2,1^4]+[p-4,2^3]$
\end{remark}

The previous lemma, together with Theorem~\ref{thm:L2calc} produces lots of representations in the Johnson cokernel:
\begin{corollary}
We have
\begin{enumerate}
\item $[2k-1,1^2]_{\SP}\otimes \mathcal S_{2k+2}\hookrightarrow \mathsf{C}_{2k+5}$ for $k\geq 1$.
\item $([2k+1,1^2]_{\SP}\oplus[2k,2,1]_{\SP}\oplus[2k,1^3]_{\SP})\otimes \mathcal M_{2k+2}\hookrightarrow\mathsf{C}_{2k+7}.$
\item $\mathcal S_{2k-2\ell+2}\otimes([2k+1,2\ell-1,1]\oplus
[2k,2\ell,1] \oplus
[2k-1,2\ell+1,1] \oplus
[2k,2\ell-1,1^2] \oplus
[2k-1,2\ell,1^2])\hookrightarrow \mathsf{C}_{2k+2\ell+5}$ for $k> \ell>0$.
\item $\mathcal M_{2k-2\ell+2}\otimes([2k+1,2\ell-1,1]\oplus
[2k+1,2\ell+1,1] \oplus
[2k,2\ell+2,1] \oplus
[2k+1,2\ell,1^2] \oplus
[2k,2\ell+1,1^2])\hookrightarrow \mathsf{C}_{2k+2\ell+7}$ for $k> \ell>0$.
\end{enumerate}
\end{corollary}

Recall that $\mathcal A_k$ is the set of Young diagrams with $2k$ boxes and an even number of boxes in each column.

\begin{proposition}
There is an injection
$$
\bigoplus_{\lambda\in \mathcal B_{m}} \SF{\lambda}(V)\hookrightarrow \overline{\SF{(m)}(\sL_{(2)})},
$$
where $\mathcal B_{m}$ is the set of Young diagrams formed by placing the diagram for a partition $(a)$ (where $1\leq a\leq m$) directly over a partition $\mu\in\mathcal A_{m-a}$ if that forms a legal Young diagram.
\end{proposition}
\begin{proof}
We have $\displaystyle \Sym^{m}(\sL_{(2)})\cong\bigoplus_{0\leq a\leq m} S_{a,m-a}$, where $S_{a,b}=\Sym^{a}(V)\otimes \Sym^{b}(\ext^2 V)$.
By the Pieri rule, the diagrams in $\mathcal B_{m}$ formed from putting $(a)$ over a diagram from $\mathcal A_{2k-a}$ appear in $S_{a,b}$.
The adjoint map decomposes as $V\otimes S_{a+1,b-1}\to S_{a,b}$, and the Young diagrams in $V\otimes S_{a+1,b-1}$ are all at least $a+1$-wide. Hence they can't appear in $\mathcal B_{m}$.
\end{proof}

\begin{corollary}
Let $\lambda\in \mathcal A_k$ be a Young diagram, say $\lambda=(\lambda_1,\ldots,\lambda_\ell)$. For all $m\geq \lambda_1$, we have that
$$[m,\lambda_1,\ldots,\lambda_\ell]\otimes \mathcal S_{m+k+2}\hookrightarrow \mathsf{C}_{m+k+2}.$$
\end{corollary}

\end{document}